\documentclass[12pt,a4paper]{amsart}
\usepackage{amsmath,amssymb,amsthm}
\usepackage{array}
\usepackage{graphicx}
\usepackage{placeins}
\usepackage{flafter}
\usepackage{natbib}
\usepackage{xcolor}
\usepackage{float}
\definecolor{linknavy}{RGB}{0,51,102}
\usepackage{hyperref}

\newtheorem{theorem}{Theorem}[section]
\newtheorem{corollary}[theorem]{Corollary}
\newtheorem{conjecture}[theorem]{Conjecture}
\theoremstyle{remark}
\newtheorem{remark}{Remark}[section]
\newcommand{\C}{\mathbb{C}}
\newcommand{\R}{\mathbb{R}}

\newcommand{\Z}{\mathbb{Z}}

\hypersetup{
    colorlinks=true,
    linkcolor=linknavy,
    citecolor=linknavy,
    filecolor=linknavy,
    urlcolor=linknavy,
    pdftitle={The Exponent Set: A Third Natural Extension of the Mandelbrot--Julia Framework},
    pdfauthor={Yuchen Brian Shen},
    pdfsubject={Exponent-parameter boundedness in the Mandelbrot--Julia framework},
    pdfkeywords={Mandelbrot set; Julia set; Exponent Set; complex exponent;
    complex-power iteration; boundedness locus; exponent parameter}
}

\title[The Exponent Set]{The Exponent Set: A Third Natural Extension of the Mandelbrot--Julia Framework}
\author{Yuchen Brian Shen}
\date{\today}
\keywords{Mandelbrot set; Julia set; Exponent Set; complex exponent;
complex-power iteration; boundedness locus; exponent parameter}

\begin{document}
\raggedbottom

\begin{abstract}
The Mandelbrot and Julia sets, generated by the quadratic iteration
\(z_{n+1}=z_n^2+c\), are foundational objects in complex
dynamics. We study the three-variable principal-value complex-power iteration
\[
z_{n+1}=z_n^{\,x}+c,
\qquad z_0,c,x\in\C.
\]
The triples producing all-time well-defined and bounded orbits form a locus
\(\mathcal{B}\subset\C^3\). Fixing two coordinates yields three natural
families of coordinate fibers, denoted \(M(z_0,x)\), \(J(c,x)\), and
\(E(z_0,c)\). For \(x=2\), \(M(0,2)\) is the classical
Mandelbrot set, \(J(c,2)\) is the classical filled Julia set, and
\(\partial J(c,2)\) is the classical Julia set.

We focus on the Exponent Set, or E-Set, obtained by fixing $(z_0,c)$ and
varying the complex exponent $x$. We prove three groups of structural results.
First, for explicit parameter families, including pure-power real and
unit-circle cases and an additive example with nonzero real and imaginary
parts, no finite universal escape radius exists: for every prescribed radius,
one can choose an exponent whose bounded orbit makes a finite excursion beyond
that radius. Second, we construct a boundary point at which an
extended-valued escape-time function is discontinuous for one strict threshold.
Third, we prove a vertical boundedness asymmetry in which the
principal-argument convention enters explicitly.
\end{abstract}

\maketitle

\section{Introduction}

The classical Mandelbrot set is defined by the quadratic iteration
$z_{n+1}=z_n^2+c$. A complex parameter $c$ belongs to the Mandelbrot set
precisely when the orbit starting from $z_0=0$ remains bounded
\cite{mandelbrot1982fractal}. Its boundary is a central object of fractal
geometry and complex dynamics \cite{falconer2003fractal}.

For the same quadratic family, fixing $c$ and varying the initial value $z_0$
produces the filled Julia set: the set of initial conditions with bounded
orbits. The Julia set is its boundary. Thus the Mandelbrot and Julia viewpoints
play different roles. The former is a parameter-plane boundedness locus, while
the latter is a dynamical-plane boundary arising from an initial-condition
boundedness locus.

Generalized Mandelbrot and Julia-type constructions modify the underlying
iteration. Huang studied global structure for higher-degree maps
\cite{huang1991global}. Other work investigated maps of the form
$z\mapsto z^\alpha+c$ with real exponents \cite{gujar1992fractal}, negative
exponents \cite{dhurandhar1993analysis}, inverse exponent maps
\cite{shirriff1993investigation}, and rational exponents
\cite{liu2013fractal}. Chen and Zhu considered complex-valued exponents in
maps of the form $z\mapsto z^w+c$ and formulated related Mandelbrot--Julia
conjectures \cite{chen1998m}; Liu et al.\ studied fractal properties and
symmetries of generalized Mandelbrot sets with complex exponents
\cite{liu2024fractal}.

These studies motivate treating the exponent as a genuine coordinate of the
full iterative family. We fix $(z_0,c)$ and analyze the resulting $x$-fibers,
restricting our claims to the explicit families treated below.

We study
\[
z_{n+1}=z_n^{\,x}+c,
\qquad z_0,c,x\in\C,
\]
using the principal-value complex-power convention specified in
Section~\ref{sec:definitions-framework}. Here $z_0$ specifies the initial condition, $c$ the additive parameter, and $x$ the complex exponent. 
The triples producing all-time well-defined and bounded orbits form a locus
$\mathcal{B}\subset\C^3$. Fixing two coordinates produces three families of
coordinate fibers: a $c$-plane boundedness locus, a $z_0$-plane
initial-condition boundedness locus, and an $x$-plane exponent-parameter
boundedness locus. We denote these by $M(z_0,x)$, $J(c,x)$, and $E(z_0,c)$,
respectively.

For $x=2$, this framework contains the classical objects: $M(0,2)$ is the
Mandelbrot set, $J(c,2)$ is the filled Julia set, and
$\partial J(c,2)$ is the Julia set. Fixing $(z_0,c)$ and varying $x$ yields the Exponent Set $E(z_0,c)$, the
third natural coordinate fiber of the Mandelbrot--Julia boundedness framework.

The central object of this paper is the Exponent Set $E(z_0,c)$. We prove the
failure of universal escape radii in explicit parameter families, a
threshold-sensitive discontinuity of an extended-valued escape-time function,
and a principal-argument-dependent vertical asymmetry in boundedness. The
numerical section records only finite-time outcomes and states the separate
all-time conjectures explicitly.

\section{Definitions and Framework}
\label{sec:definitions-framework}

\subsection{Principal-Value Complex Powers}

Throughout this paper, for $z\in\C\setminus\{0\}$ we use the single-valued
principal-value logarithm
\[
\log z=\ln|z|+i\operatorname{Arg}(z),
\qquad
\operatorname{Arg}(z)\in(-\pi,\pi].
\]
This is a single-valued principal-value convention on
$\C\setminus\{0\}$ whose jump discontinuity lies along the negative real
axis; it is not a globally holomorphic logarithm branch on
$\C\setminus\{0\}$.

We also consider alternative logarithm branches defined on specified cut
domains. For $\alpha\in\R$, let
\[
D_\alpha
=
\C\setminus\{re^{i\alpha}:r\ge0\},
\]
and let $\operatorname{Arg}_\alpha(z)\in(\alpha-2\pi,\alpha)$ denote the
continuous argument on $D_\alpha$. For each $k\in\Z$, define
\[
\log_{\alpha,k}z
=
\ln|z|+i\operatorname{Arg}_\alpha(z)+2\pi i k,
\qquad z\in D_\alpha.
\]
The associated complex power is $\exp(x\log_{\alpha,k}z)$. Any comparison
with this branch applies only while every nonzero base in the orbit remains in
$D_\alpha$; it is not a globally defined logarithm on $\C\setminus\{0\}$.
For $z\ne 0$ and $x\in\C$, the complex power $z^x$ is defined by
\[
z^x=\exp(x\log(z)).
\]

The point $z=0$ requires a separate convention because $\log(0)$ is undefined.
We use the following convention throughout the paper:
\begin{enumerate}
    \item If $z=0$ and $\operatorname{Re}(x)>0$, define $0^x=0$.
    \item If $z=0$ and $\operatorname{Re}(x)\le 0$, then $0^x$ is left
    undefined. Any orbit that reaches $0$ at such an exponent is excluded
    from the corresponding boundedness set.
\end{enumerate}

Under the principal-argument convention, points on the negative real axis have
argument $\pi$; that is, $\operatorname{Arg}(-a)=\pi$ for $a>0$. This choice makes the
iteration single-valued and deterministic, but it also introduces
discontinuity effects along the negative real axis that become important in
the dynamics of the Exponent Set.

\subsection{The Three Coordinate-Fiber Families and the Exponent Set}

For each triple $(z_0,c,x)\in\C^3$, let $\{z_n\}$ denote the maximal orbit
generated recursively by
\[
z_{n+1}=z_n^{\,x}+c
\]
under the convention above: the recursion continues for as long as the next
iterate is defined and terminates otherwise. Define the full boundedness locus
\[
\mathcal{B}
=
\left\{(z_0,c,x)\in\C^3:
\begin{array}{l}
z_n\text{ is defined for every }n\ge0,\\
\displaystyle\sup_{n\ge0}|z_n|<\infty
\end{array}
\right\}.
\]
Fixing two variables produces three families of coordinate fibers.

\begin{enumerate}
    \item \textbf{The $c$-plane.}
    For fixed $z_0,x\in\C$, define
    \[
    M(z_0,x)=\{c\in\C:(z_0,c,x)\in\mathcal{B}\}.
    \]
    This is a Mandelbrot-type parameter locus. In particular, $M(0,2)$ is the
    classical Mandelbrot set.

    \item \textbf{The $z_0$-plane.}
    For fixed $c,x\in\C$, define
    \[
    J(c,x)=\{z_0\in\C:(z_0,c,x)\in\mathcal{B}\}.
    \]
    This is a filled Julia-type boundedness locus. Its boundary
    $\partial J(c,x)$ is the corresponding Julia-type boundary. For $x=2$,
    $J(c,2)$ is the classical filled Julia set, while $\partial J(c,2)$ is the
    classical Julia set. Outside the polynomial setting, this terminology does
    not assert the full structure of ordinary holomorphic Julia theory.

    \item \textbf{The $x$-plane.}
    For fixed $z_0,c\in\C$, define the \emph{Exponent Set}, or \emph{E-Set}, by
    \[
    E(z_0,c)=\{x\in\C:(z_0,c,x)\in\mathcal{B}\}.
    \]
    Equivalently, $x\in E(z_0,c)$ exactly when the orbit is defined at every
    iteration and remains bounded.
\end{enumerate}

The Exponent Set is therefore the exponent-coordinate fiber of the same
boundedness condition that produces the Mandelbrot-type and filled Julia-type
loci. The remainder of the paper focuses on $E(z_0,c)$.

\section{Theorems of the E-Set}
This section organizes six theorems and one corollary around three structural
phenomena of the Exponent Set: failure of a universal escape radius in
explicit parameter classes, threshold-sensitive discontinuity of an
extended-valued escape-time function, and principal-argument-induced directional
asymmetry in boundedness.

\subsection{Failure of Universal Escape Radii}

The first group of results shows that the classical escape-radius principle
does not extend directly to the exponent-parameter setting. In the classical quadratic Mandelbrot family, crossing a sufficiently large
radius certifies escape. In the explicit parameter classes considered below, this can
fail even for bounded orbits: for every prescribed finite radius, one can
choose an exponent, depending on that radius, whose orbit exceeds the radius
while remaining bounded.

We first prove this phenomenon for the pure-power slice $c=0$. The unit-circle
case gives a simple exact-collapse construction. The positive and negative real
cases then show that the failure persists for every real initial value except
the degenerate cases $z_0=0$ and $z_0=1$. Finally, we give an additive example
in which $z_0$ and $c$ have nonzero real and imaginary parts, showing that the
phenomenon is not restricted to the pure-power slice.

\noindent\textbf{Universal Escape Radius.}
Fix $z_0,c\in\C$. A finite number $r>0$ is called a
\emph{universal escape radius} for $E(z_0,c)$ if, for every exponent
$x\in\C$ for which the orbit is well-defined for all $n\ge 0$,
\[
(\exists n\ge 0:\; |z_n|>r)
\quad\Longrightarrow\quad
x\notin E(z_0,c),
\]
where the sequence $\{z_n\}$ is generated by
$z_{n+1}=z_n^{\,x}+c$ with initial value $z_0$.

The adjective \emph{universal} means that $r$ is independent of the exponent:
for the fixed pair $(z_0,c)$, the same number must apply across the entire
complex exponent plane to every orbit that is well-defined for all time.
Equivalently, every $x\in E(z_0,c)$ must satisfy $|z_n|\le r$ for all
$n\ge0$. Failure of such a global uniform radius does not rule out a bound on
a fixed compact exponent viewport, a bound for a restricted exponent family,
or a useful finite-time threshold in a numerical computation.

\begin{theorem}[Absence of a Universal Escape Radius on the Unit Circle]
\label{thm:no-universal-radius-unit-circle}
Let $c=0$ and let $z_0\in\C$ satisfy
\[
|z_0|=1,
\qquad
z_0\ne 1.
\]
Then the Exponent Set $E(z_0,0)$ admits no finite universal escape radius.

More precisely, for every $r>0$, there exists an exponent $x\in E(z_0,0)$
such that the orbit generated by
\[
z_{n+1}=z_n^{\,x}
\]
satisfies $|z_n|>r$ for some $n\ge 0$, while the orbit remains bounded.
\end{theorem}

\begin{proof}
Since $|z_0|=1$ and $z_0\ne 1$, there exists
\[
\theta=\operatorname{Arg}(z_0)\in(-\pi,\pi]
\]
with $\theta\ne 0$ such that
\[
z_0=e^{i\theta}.
\]
By the principal-argument convention,
\[
\log(z_0)=i\theta.
\]
The case $z_0=-1$ corresponds to $\theta=\pi$, so it is included in this
argument.

Let $k\ge 1$ be an integer and define
\[
M_k=\sqrt{2\pi k|\theta|}.
\]
Now define
\[
x_k=\frac{M_k}{i\theta}.
\]
Since $\theta\ne 0$, this exponent is well-defined. Equivalently,
\[
x_k=-\frac{iM_k}{\theta},
\]
so $x_k$ is purely imaginary.

We first compute the first iterate. Using $\log(z_0)=i\theta$, we obtain
\[
z_1
=
z_0^{\,x_k}
=
\exp(x_k\log(z_0))
=
\exp\left(\frac{M_k}{i\theta}i\theta\right)
=
\exp(M_k).
\]
Thus
\[
|z_1|=e^{M_k}.
\]
Since $M_k\to+\infty$ as $k\to\infty$, for every $r>0$ we can choose $k$
large enough such that
\[
|z_1|=e^{M_k}>r.
\]
Therefore the orbit exceeds the radius $r$ at the first iteration step.

It remains to prove that the same orbit remains bounded. Since
\[
z_1=e^{M_k}>0,
\]
the principal-value logarithm of $z_1$ agrees with the real natural logarithm:
\[
\log(z_1)=\ln(z_1)=M_k.
\]
Therefore
\[
z_2
=
z_1^{\,x_k}
=
\exp(x_k\log(z_1))
=
\exp\left(\frac{M_k}{i\theta}M_k\right)
=
\exp\left(\frac{M_k^2}{i\theta}\right).
\]
Since $1/i=-i$, this becomes
\[
z_2
=
\exp\left(-i\frac{M_k^2}{\theta}\right).
\]
By definition,
\[
M_k^2=2\pi k|\theta|.
\]
Hence
\[
-\frac{M_k^2}{\theta}
=
-2\pi k\frac{|\theta|}{\theta}.
\]

If $\theta>0$, then $|\theta|/\theta=1$, so
\[
-\frac{M_k^2}{\theta}=-2\pi k.
\]
Thus
\[
z_2=\exp(-2\pi i k)=1.
\]

If $\theta<0$, then $|\theta|/\theta=-1$, so
\[
-\frac{M_k^2}{\theta}=2\pi k.
\]
Thus
\[
z_2=\exp(2\pi i k)=1.
\]

Therefore, in all cases,
\[
z_2=1.
\]
Now, using the principal-value logarithm, $\log(1)=0$. Hence
\[
z_3
=
1^{x_k}
=
\exp(x_k\log(1))
=
\exp(0)
=
1.
\]
By induction,
\[
z_n=1
\qquad\text{for all } n\ge 2.
\]
Therefore the orbit is
\[
z_0,\quad e^{M_k},\quad 1,\quad 1,\quad 1,\ldots
\]
and is bounded by
\[
B_k=\max\{1,e^{M_k}\}<\infty.
\]
All iterates in this constructed orbit are nonzero, so the orbit is
well-defined for every $n\ge 0$. Hence
\[
x_k\in E(z_0,0).
\]

Since for every $r>0$ we can choose $k$ such that the bounded orbit
corresponding to $x_k$ satisfies $|z_1|>r$, no finite number $r$ can serve as
a universal escape radius for $E(z_0,0)$.
\end{proof}

\medskip

\begin{theorem}[Absence of a Universal Escape Radius for Positive Real Initial Values]
\label{thm:no-universal-radius-positive-real}
Let $c=0$ and let $z_0\in\R$ satisfy
\[
z_0>0,
\qquad
z_0\ne 1.
\]
Then the Exponent Set $E(z_0,0)$ admits no finite universal escape radius.

More precisely, for every $r>0$, there exists an exponent $x\in E(z_0,0)$
such that the orbit generated by
\[
z_{n+1}=z_n^{\,x}
\]
satisfies $|z_n|>r$ for some $n\ge 0$, while the orbit remains bounded.
\end{theorem}

\begin{proof}
Let
\[
L=\ln(z_0).
\]
Since $z_0>0$ and $z_0\ne 1$, we have
\[
L\in\R\setminus\{0\}.
\]
Moreover, because $z_0$ is positive real, the principal-value logarithm satisfies
\[
\log(z_0)=L.
\]

Fix $r>0$. We will construct an exponent $x$ such that the orbit satisfies
\[
z_0\mapsto z_1\mapsto 1\mapsto 1\mapsto \cdots
\]
and such that $|z_1|>r$.

For each integer $p\ge 1$, define
\[
H_p(\theta)
=
\frac{\sqrt{\theta(\theta+2\pi p)}(2\theta+2\pi p)}
{2\pi |L|}
\]
for
\[
\theta\in[\pi/2,\pi].
\]
The function $H_p$ is continuous on $[\pi/2,\pi]$ because
$\theta(\theta+2\pi p)>0$ on this interval, and products, quotients by
nonzero constants, and square-roots of positive continuous functions are
continuous.

At the endpoints, we have
\[
H_p(\pi)
=
\frac{\pi(p+1)\sqrt{2p+1}}{|L|}
\]
and
\[
H_p(\pi/2)
=
\frac{\pi(2p+1)\sqrt{4p+1}}{4|L|}.
\]
Moreover,
\[
\frac{H_p(\pi)}{p^{3/2}}
\to
\frac{\pi\sqrt{2}}{|L|},
\qquad
\frac{H_p(\pi/2)}{p^{3/2}}
\to
\frac{\pi}{|L|}
\]
as $p\to\infty$. Since $\sqrt{2}>1$, it follows that
\[
H_p(\pi)-H_p(\pi/2)\to+\infty
\qquad\text{as }p\to\infty.
\]
Also, for every $\theta\in[\pi/2,\pi]$,
\[
\sqrt{\theta(\theta+2\pi p)}
\ge
\sqrt{\frac{\pi}{2}(\frac{\pi}{2}+2\pi p)}
\to+\infty
\qquad\text{as }p\to\infty.
\]
Therefore we may choose $p$ sufficiently large such that
\[
H_p(\pi)-H_p(\pi/2)>1
\]
and
\[
\sqrt{\frac{\pi}{2}(\frac{\pi}{2}+2\pi p)}
>
\max\{0,\ln(r)\}.
\]

Define
\[
m=\left\lceil H_p(\pi/2)\right\rceil.
\]
Because $H_p(\pi)-H_p(\pi/2)>1$ and both endpoint values are positive, we
have
\[
1\le m\le H_p(\pi).
\]
By the intermediate value theorem, there exists
\[
\theta\in[\pi/2,\pi]
\]
such that
\[
H_p(\theta)=m.
\]

Now define
\[
T=\theta+2\pi p
\]
and
\[
M=\sqrt{\theta T}
=
\sqrt{\theta(\theta+2\pi p)}.
\]
By the choice of $p$ and because $\theta\in[\pi/2,\pi]$, we have
\[
M>\max\{0,\ln(r)\}.
\]
In particular, $M>\ln(r)$ if $r\ge 1$, while $M>0$ if $0<r<1$. Hence in both
cases,
\[
e^M>r.
\]

Finally, define
\[
x=\frac{M+iT}{L}.
\]
This is well-defined because $L\ne 0$.

We now compute the orbit. Since $\log(z_0)=L$, the first iterate is
\[
z_1
=
z_0^{\,x}
=
\exp(x\log(z_0))
=
\exp(xL)
=
\exp(M+iT).
\]
Because $T=\theta+2\pi p$, this gives
\[
z_1=e^M e^{i\theta}.
\]
Therefore
\[
|z_1|=e^M>r.
\]

It remains to prove that the orbit is bounded. Since
\[
\theta\in[\pi/2,\pi]\subset(-\pi,\pi],
\]
the argument $\theta$ lies in the principal range. Therefore the principal
logarithm of $z_1=e^M e^{i\theta}$ is
\[
\log(z_1)
=
\ln(|z_1|)+i\operatorname{Arg}(z_1)
=
M+i\theta.
\]
Thus
\[
z_2
=
z_1^{\,x}
=
\exp(x\log(z_1))
=
\exp\left(\frac{M+iT}{L}(M+i\theta)\right).
\]
Expanding the product gives
\[
(M+iT)(M+i\theta)
=
M^2-T\theta+iM(T+\theta).
\]
By the definition $M^2=\theta T$, the real part vanishes:
\[
M^2-T\theta=0.
\]
Therefore
\[
x\log(z_1)
=
i\frac{M(T+\theta)}{L}.
\]

By the definition of $H_p$ and the choice $H_p(\theta)=m$, we have
\[
m
=
H_p(\theta)
=
\frac{\sqrt{\theta(\theta+2\pi p)}(2\theta+2\pi p)}
{2\pi |L|}.
\]
Since $M=\sqrt{\theta(\theta+2\pi p)}$ and $T=\theta+2\pi p$, this becomes
\[
m
=
\frac{M(T+\theta)}{2\pi |L|}.
\]
Equivalently,
\[
M(T+\theta)=2\pi m|L|.
\]
Hence
\[
\frac{M(T+\theta)}{L}
=
2\pi m\frac{|L|}{L}.
\]
Since $L\ne 0$, the number
\[
q=m\frac{|L|}{L}
\]
is an integer, equal to $m$ if $L>0$ and equal to $-m$ if $L<0$. Thus
\[
x\log(z_1)=2\pi i q.
\]
Therefore
\[
z_2=\exp(2\pi i q)=1.
\]

Now, using the principal-value logarithm, $\log(1)=0$. Hence
\[
z_3
=
1^x
=
\exp(x\log(1))
=
\exp(0)
=
1.
\]
By induction,
\[
z_n=1
\qquad\text{for all } n\ge 2.
\]
Thus the orbit is
\[
z_0,\quad z_1,\quad 1,\quad 1,\quad 1,\ldots
\]
and is bounded by
\[
B=\max\{|z_0|,|z_1|,1\}<\infty.
\]
All iterates in this constructed orbit are nonzero, so the orbit is
well-defined for every $n\ge 0$. Therefore
\[
x\in E(z_0,0).
\]

Since for the arbitrary radius $r>0$ we have constructed an exponent
$x\in E(z_0,0)$ whose orbit satisfies $|z_1|>r$, no finite number $r$ can
serve as a universal escape radius for $E(z_0,0)$.
\end{proof}

\begin{remark}[Positive real case]
Because $\log(z_0)=\ln(z_0)$ is real, the imaginary part of $x$ supplies the
winding needed to make $z_1$ arbitrarily large while enforcing $z_2=1$.
\end{remark}

\medskip

\begin{theorem}[Absence of a Universal Escape Radius for Negative Real Initial Values]
\label{thm:no-universal-radius-negative-real}
Let $c=0$ and let $z_0\in\R$ satisfy
\[
z_0<0,
\qquad
z_0\ne -1.
\]
Then the Exponent Set $E(z_0,0)$ admits no finite universal escape radius.

More precisely, for every $r>0$, there exists an exponent $x\in E(z_0,0)$
such that the orbit generated by
\[
z_{n+1}=z_n^{\,x}
\]
satisfies $|z_n|>r$ for some $n\ge 0$, while the orbit remains bounded.
\end{theorem}

\begin{proof}
Write
\[
z_0=-\rho,
\qquad
\rho>0.
\]
Since $z_0\ne -1$, we have $\rho\ne 1$. Define
\[
a=\ln(\rho).
\]
Then
\[
a\in\R\setminus\{0\}.
\]
By the principal-argument convention, the principal argument of any negative
real number is $\pi$. Therefore
\[
\log(z_0)=a+i\pi.
\]
Here $\pi$, not $-\pi$, is the principal argument of $z_0<0$.

Let
\[
A=|a|.
\]
Then $A>0$. Define
\[
\sigma=-\frac{a}{A}.
\]
Then $\sigma\in\{-1,1\}$ and
\[
\sigma A=-a.
\]

Fix $r>0$. We will construct an exponent $x$ such that the orbit satisfies
\[
z_0\mapsto z_1\mapsto 1\mapsto 1\mapsto \cdots
\]
and such that $|z_1|>r$.

For each integer $N\ge 1$ and each
\[
u\in[\pi/2,3\pi/4],
\]
define
\[
S_N(u)=u+\pi N,
\]
\[
B_N(u)=\frac{2\pi S_N(u)}{A},
\]
and
\[
C_N(u)=u(u+2\pi N).
\]
Now define
\[
M_N(u)
=
\frac{B_N(u)+\sqrt{B_N(u)^2+4C_N(u)}}{2}.
\]
Then $M_N(u)>0$, and $M_N(u)$ is the positive root of
\[
M^2-B_N(u)M-C_N(u)=0.
\]
Equivalently,
\[
M_N(u)^2-C_N(u)=B_N(u)M_N(u).
\]

Define
\[
G_N(u)=\frac{M_N(u)S_N(u)}{\pi A}.
\]
The displayed radical formula shows that $M_N$ is continuously
differentiable on $[\pi/2,3\pi/4]$. Consequently, $G_N$ is continuously
differentiable there as well.

We first show that, for some large $N$, the function $G_N$ takes an integer
value. Since $B_N(u)>0$ and $C_N(u)>0$, we have
\[
M_N(u)>B_N(u)
\]
for all $u\in[\pi/2,3\pi/4]$.

Moreover, differentiating the equation
\[
M_N(u)^2-B_N(u)M_N(u)-C_N(u)=0
\]
gives
\[
(2M_N(u)-B_N(u))M_N'(u)=B_N'(u)M_N(u)+C_N'(u).
\]
Here
\[
2M_N(u)-B_N(u)=\sqrt{B_N(u)^2+4C_N(u)}>0,
\]
while $B_N'(u)>0$ and $C_N'(u)>0$. Hence
\[
M_N'(u)>0.
\]
Therefore
\[
G_N'(u)
=
\frac{M_N'(u)S_N(u)+M_N(u)}{\pi A}
>
\frac{M_N(u)}{\pi A}
>
\frac{B_N(u)}{\pi A}
=
\frac{2S_N(u)}{A^2}.
\]
It follows that
\[
G_N(3\pi/4)-G_N(\pi/2)
>
\int_{\pi/2}^{3\pi/4}\frac{2(u+\pi N)}{A^2}\,du.
\]
The right-hand side tends to $+\infty$ as $N\to\infty$. Hence we may choose
$N$ sufficiently large such that
\[
G_N(3\pi/4)-G_N(\pi/2)>1.
\]
We also choose $N$ sufficiently large so that
\[
M_N(u)>\max\{0,\ln(r)\}
\]
for every $u\in[\pi/2,3\pi/4]$. This is possible because
\[
M_N(u)>B_N(u)=\frac{2\pi(u+\pi N)}{A}\to+\infty
\]
uniformly for $u\in[\pi/2,3\pi/4]$.

Define
\[
m=\left\lceil G_N(\pi/2)\right\rceil.
\]
Because $G_N(3\pi/4)-G_N(\pi/2)>1$ and both endpoint values are positive, we
have
\[
1\le m\le G_N(3\pi/4).
\]
By the intermediate value theorem, there exists
\[
u\in[\pi/2,3\pi/4]
\]
such that
\[
G_N(u)=m.
\]

Now set
\[
M=M_N(u),
\qquad
S=S_N(u),
\]
and define
\[
\theta=\sigma u,
\qquad
p=\sigma N,
\qquad
T=\theta+2\pi p.
\]
Since $\sigma=\pm 1$, we have $p\in\Z$. Also,
\[
T=\sigma(u+2\pi N).
\]
Because $u\in[\pi/2,3\pi/4]$, we have
\[
\theta\in[-3\pi/4,-\pi/2]\cup[\pi/2,3\pi/4]\subset(-\pi,\pi].
\]
Thus $\theta$ lies in the principal argument range.

Finally, define
\[
x=\frac{M+iT}{a+i\pi}.
\]
This is well-defined because $a+i\pi\ne 0$.

We now compute the orbit. Since $\log(z_0)=a+i\pi$, the first iterate is
\[
z_1
=
z_0^{\,x}
=
\exp(x\log(z_0))
=
\exp(M+iT).
\]
Since $T=\theta+2\pi p$ and $p\in\Z$, this becomes
\[
z_1=e^M e^{i\theta}.
\]
Therefore
\[
|z_1|=e^M.
\]
By the choice of $N$, we have $M>\max\{0,\ln(r)\}$. Hence
\[
|z_1|=e^M>r.
\]

It remains to prove that the orbit is bounded. Since $\theta\in(-\pi,\pi]$,
the principal-value logarithm of $z_1$ is
\[
\log(z_1)=M+i\theta.
\]
Thus
\[
z_2
=
z_1^{\,x}
=
\exp(x\log(z_1))
=
\exp\left(\frac{M+iT}{a+i\pi}(M+i\theta)\right).
\]

We now compute the exponent explicitly. Since
\[
T=\sigma(u+2\pi N)
\qquad\text{and}\qquad
\theta=\sigma u,
\]
we have
\[
T\theta
=
\sigma(u+2\pi N)\sigma u
=
u(u+2\pi N)
=
C_N(u).
\]
Also,
\[
T+\theta
=
\sigma(u+2\pi N)+\sigma u
=
2\sigma(u+\pi N)
=
2\sigma S.
\]

Since $M=M_N(u)$ is the positive root of
\[
M^2-B_N(u)M-C_N(u)=0,
\]
we have
\[
M^2-C_N(u)=B_N(u)M.
\]
Using $T\theta=C_N(u)$, this gives
\[
M^2-T\theta=B_N(u)M.
\]
Because
\[
B_N(u)=\frac{2\pi S}{A},
\]
we obtain
\[
M^2-T\theta
=
\frac{2\pi SM}{A}.
\]
On the other hand, from
\[
G_N(u)=\frac{M_N(u)S_N(u)}{\pi A}=m
\]
and the definitions $M=M_N(u)$ and $S=S_N(u)$, we get
\[
MS=\pi A m.
\]
Therefore
\[
M^2-T\theta
=
\frac{2\pi MS}{A}
=
2\pi^2m.
\]

Similarly, since $T+\theta=2\sigma S$, we have
\[
M(T+\theta)
=
2\sigma MS.
\]
Using $MS=\pi A m$, this becomes
\[
M(T+\theta)
=
2\sigma\pi A m.
\]
Since $\sigma A=-a$, we obtain
\[
M(T+\theta)
=
-2\pi a m.
\]

Therefore
\[
\begin{aligned}
(M+iT)(M+i\theta)
&=
M^2-T\theta+iM(T+\theta) \\
&=
2\pi^2m-2\pi i a m.
\end{aligned}
\]
The last expression factors as
\[
2\pi^2m-2\pi i a m
=
-2\pi i m(a+i\pi).
\]
Hence
\[
(M+iT)(M+i\theta)
=
-2\pi i m(a+i\pi).
\]
Since
\[
x=\frac{M+iT}{a+i\pi}
\qquad\text{and}\qquad
\log(z_1)=M+i\theta,
\]
we conclude that
\[
x\log(z_1)
=
\frac{M+iT}{a+i\pi}(M+i\theta)
=
-2\pi i m.
\]
Therefore
\[
z_2
=
\exp(x\log(z_1))
=
\exp(-2\pi i m)
=
1.
\]

Now, using the principal-value logarithm, $\log(1)=0$. Hence
\[
z_3
=
1^x
=
\exp(x\log(1))
=
\exp(0)
=
1.
\]
By induction,
\[
z_n=1
\qquad\text{for all }n\ge 2.
\]
Thus the orbit is
\[
z_0,\quad z_1,\quad 1,\quad 1,\quad 1,\ldots
\]
and is bounded by
\[
B=\max\{|z_0|,|z_1|,1\}<\infty.
\]
All iterates in this constructed orbit are nonzero, so the orbit is
well-defined for every $n\ge 0$. Therefore
\[
x\in E(z_0,0).
\]

Since for the arbitrary radius $r>0$ we have constructed an exponent
$x\in E(z_0,0)$ whose orbit satisfies $|z_1|>r$, no finite number $r$ can
serve as a universal escape radius for $E(z_0,0)$.
\end{proof}

\begin{remark}[Negative real case]
Here $\log(z_0)=\ln|z_0|+i\pi$. The exponent compensates for the additional
$i\pi$ term; after that adjustment, the same large-excursion/exact-collapse
mechanism applies.
\end{remark}

\medskip

\begin{corollary}[Absence of a Universal Escape Radius on the Real Axis]
\label{cor:no-universal-radius-real-axis}
Let $c=0$ and let $z_0\in\R$ satisfy
\[
z_0\ne 0,
\qquad
z_0\ne 1.
\]
Then the Exponent Set $E(z_0,0)$ admits no finite universal escape radius.

More precisely, for every $r>0$, there exists an exponent $x\in E(z_0,0)$
such that the orbit generated by
\[
z_{n+1}=z_n^{\,x}
\]
satisfies $|z_n|>r$ for some $n\ge 0$, while the orbit remains bounded.
\end{corollary}

\begin{proof}
Let $z_0\in\R$ with $z_0\ne 0$ and $z_0\ne 1$.

If $z_0>0$, then the result follows from
Theorem~\ref{thm:no-universal-radius-positive-real}.

If $z_0=-1$, then $|z_0|=1$ and $z_0\ne 1$, so the result follows from
Theorem~\ref{thm:no-universal-radius-unit-circle}.

If $z_0<0$ and $z_0\ne -1$, then the result follows from
Theorem~\ref{thm:no-universal-radius-negative-real}.

These cases exhaust all real values $z_0\ne 0,1$. Therefore
$E(z_0,0)$ admits no finite universal escape radius for every
$z_0\in\R\setminus\{0,1\}$.
\end{proof}

\begin{remark}[Exceptional initial values]
For $c=0$, the excluded values are degenerate: $z_0=1$ is fixed, while every
well-defined orbit starting at $z_0=0$ remains at $0$.
\end{remark}

\smallskip

\begin{theorem}[An Additive Example with Nonzero Real and Imaginary Parts]
\label{thm:no-universal-radius-nonzero-components}
There exist parameters $z_0,c\in\C$ with
\[
\operatorname{Re}(z_0)\ne 0,
\qquad
\operatorname{Im}(z_0)\ne 0,
\qquad
\operatorname{Re}(c)\ne 0,
\qquad
\operatorname{Im}(c)\ne 0,
\]
for which the Exponent Set $E(z_0,c)$ admits no finite universal escape
radius.

In particular, for
\[
z_0=c=\frac{1+i}{10},
\]
the following holds: for every $r>0$, there exists an exponent
$x\in E(z_0,c)$ such that the orbit generated by
\[
z_{n+1}=z_n^{\,x}+c
\]
satisfies $|z_n|>r$ for some $n\ge 0$, while the orbit remains bounded.
\end{theorem}

\begin{proof}
Let
\[
s=\frac{1+i}{10}.
\]
We prove the result for
\[
z_0=c=s.
\]
Then
\[
\operatorname{Re}(z_0)=\operatorname{Im}(z_0)
=
\operatorname{Re}(c)=\operatorname{Im}(c)
=
\frac{1}{10},
\]
so all four listed real and imaginary parts are nonzero.

Let
\[
\rho=|s|=\frac{\sqrt{2}}{10}
\]
and choose
\[
\eta=\frac{\rho}{2}.
\]
Then
\[
\rho+\eta=\frac{3\rho}{2}=\frac{3\sqrt{2}}{20}<1,
\qquad
\rho-\eta=\frac{\rho}{2}>0.
\]

Fix $r>0$. We will construct a negative integer exponent $x=-N$ such that
the orbit makes a large first excursion but remains trapped in a bounded
two-region cycle.

Let $N\ge 1$ be an integer to be chosen later and set
\[
x=-N.
\]
Since $x=-N$ is a negative integer, the principal-value power agrees with
the ordinary reciprocal integer power on $\C\setminus\{0\}$. Indeed, if
$z=\lambda e^{i\theta}$ with $\lambda=|z|>0$ and
$\theta=\operatorname{Arg}(z)\in(-\pi,\pi]$, then
\[
\begin{aligned}
z^{\,x}=z^{-N}
&=\exp(-N\log(z)) \\
&=\exp(-N(\ln(\lambda)+i\theta)) \\
&=\lambda^{-N}e^{-iN\theta} \\
&=\frac{1}{z^N}.
\end{aligned}
\]
Thus, as long as the orbit avoids $0$, the iteration becomes
\[
z_{n+1}=z_n^{-N}+s.
\]

Define the closed disk
\[
A=\{z\in\C:\ |z-s|\le \eta\}.
\]
If $z\in A$, then
\[
|z|\le |s|+\eta=\rho+\eta
\]
and
\[
|z|\ge |s|-\eta=\rho-\eta>0.
\]
Therefore every point of $A$ is nonzero.

For this fixed $N$, define
\[
R_N=(\rho+\eta)^{-N}-\rho
\]
and
\[
U_N=(\rho-\eta)^{-N}+\rho.
\]
Since $\rho+\eta<1$, we have
\[
R_N\to+\infty
\qquad\text{as }N\to\infty.
\]
Moreover,
\[
R_N^{-N}\to 0
\qquad\text{as }N\to\infty.
\]
Indeed, $R_N\to+\infty$, so eventually $R_N\ge 2$, and then
\[
0<R_N^{-N}\le 2^{-N}\to0.
\]
Hence we may choose $N$ sufficiently large such that
\[
R_N>\max\{r,1\}
\]
and
\[
R_N^{-N}\le \eta.
\]

We now prove two trapping estimates.

First, suppose that $z\in A$. Then
\[
|z|\le \rho+\eta
\]
and
\[
|z|\ge \rho-\eta.
\]
Therefore
\[
|z^{-N}+s|
\ge
|z|^{-N}-|s|
\ge
(\rho+\eta)^{-N}-\rho
=
R_N,
\]
and also
\[
|z^{-N}+s|
\le
|z|^{-N}+|s|
\le
(\rho-\eta)^{-N}+\rho
=
U_N.
\]
Thus
\[
z\in A
\quad\Longrightarrow\quad
R_N\le |z^{-N}+s|\le U_N.
\]

Second, suppose that
\[
|z|\ge R_N.
\]
Because the function $t\mapsto t^{-N}$ is decreasing on $(0,\infty)$, we have
\[
|z|^{-N}\le R_N^{-N}.
\]
Hence
\[
|z^{-N}+s-s|
=
|z^{-N}|
=
|z|^{-N}
\le
R_N^{-N}
\le
\eta.
\]
Therefore
\[
|z|\ge R_N
\quad\Longrightarrow\quad
z^{-N}+s\in A.
\]

Since the initial value is
\[
z_0=s\in A,
\]
the first trapping estimate gives
\[
R_N\le |z_1|\le U_N.
\]
In particular, because $R_N>r$,
\[
|z_1|>r.
\]

The second trapping estimate then gives
\[
z_2\in A.
\]
Repeating the two implications inductively, we obtain
\[
z_{2j}\in A
\qquad\text{for all }j\ge 0,
\]
and
\[
R_N\le |z_{2j+1}|\le U_N
\qquad\text{for all }j\ge 0.
\]
Therefore the entire orbit is bounded. Indeed,
\[
|z_{2j}|\le \rho+\eta
\]
and
\[
|z_{2j+1}|\le U_N.
\]
Thus the orbit is bounded by
\[
B=\max\{\rho+\eta,U_N\}<\infty.
\]

Moreover, the orbit never reaches $0$. Points in $A$ satisfy
\[
|z|\ge \rho-\eta>0,
\]
and points in the annular region satisfy
\[
|z|\ge R_N>0.
\]
Thus the orbit is well-defined for every $n\ge 0$ under the exponent
$x=-N$.

Hence
\[
x=-N\in E(s,s).
\]
Since for the arbitrary radius $r>0$ we have constructed an exponent
$x\in E(s,s)$ whose orbit satisfies $|z_1|>r$, no finite number $r$ can serve
as a universal escape radius for $E(s,s)$.

Therefore $E(z_0,c)$ admits no finite universal escape radius for the choice
\[
z_0=c=\frac{1+i}{10}.
\]
\end{proof}

\begin{remark}[Scope of the additive example]
Unlike the preceding exact-collapse constructions, this proof uses alternating
trapping regions for one explicit diagonal pair $z_0=c$. It does not establish
persistence under parameter perturbations or a general classification.
\end{remark}

\subsection{Escape-Time and Directional Asymmetry}

The following results exhibit two precise parameter-space effects: a
threshold-sensitive discontinuity of an extended-valued escape-time function,
and opposite boundedness behavior along the two vertical directions from a
boundary point. The first effect uses equality with the chosen threshold and
the growth of nearby real-exponent orbits; the second uses the principal-argument
choice in complex exponentiation.

\begin{theorem}[Strict-Threshold Escape-Time Discontinuity at $x=1$]
\label{thm:boundary-escape-time-instability}
Consider the iterative family
\[
z_{n+1}=z_n^{\,x}+c,\qquad z_0,c,x\in\C,
\]
with the principal-value complex-power convention and the zero convention of
Section~\ref{sec:definitions-framework}. In particular, for $z\ne0$,
\[
z^x=\exp(x\log z),\qquad
\log z=\ln|z|+i\operatorname{Arg}(z),\qquad \operatorname{Arg}(z)\in(-\pi,\pi].
\]
Choose
\[
z_0=e,\qquad c=0,\qquad r=e,\qquad x_0=1.
\]
For these parameters, the following hold:
\begin{enumerate}
    \item $x_0\in\partial E(e,0)$, and the orbit of $x_0$ is bounded.
    \item For every $\delta>0$, there exists $x'\in\R$ with
    \[
    1<x'<1+\delta
    \]
    such that
    \[
    |z_1(x')|>r
    \]
    and
    \[
    z_n(x')\longrightarrow+\infty
    \qquad\text{as }n\to\infty.
    \]
\end{enumerate}
Consequently, for this particular escape threshold $r=e$, consider the
extended-valued escape-time function
\[
T_r:\C\longrightarrow[0,+\infty],
\qquad
T_r(x)=\inf\{n\in\Z_{\ge0}:\ |z_n(x)|>r\},
\]
where the infimum of the empty set is $+\infty$ and $[0,+\infty]$ carries
its usual order topology. Then $T_r$ is discontinuous at $x_0=1$.

\end{theorem}

\begin{proof}
Since $c=0$, the iteration becomes
\[
z_{n+1}=z_n^{\,x}.
\]
Because $z_0=e\ne0$ and every value $\exp(x\log z)$ is nonzero, the orbit is
well-defined for every $x\in\C$ and every iteration step.

\smallskip
\noindent\textbf{Step 1: The orbit for $x_0=1$ is bounded.}
For $x_0=1$, we have
\[
z_1(1)=e^1=e.
\]
If $z_n(1)=e$ for some $n\ge0$, then
\[
z_{n+1}(1)=z_n(1)^1=e.
\]
Therefore, by induction,
\[
z_n(1)=e
\qquad\text{for all } n\ge0.
\]
Hence
\[
\sup_{n\ge0}|z_n(1)|=e<\infty,
\]
so
\[
1\in E(e,0).
\]
Moreover, since $|z_n(1)|=e=r$ for every $n\ge0$, the orbit of $x_0=1$
never strictly exceeds the chosen escape threshold $r=e$.

\smallskip
\noindent\textbf{Step 2: Nearby real exponents larger than $1$ escape past $r=e$.}
Let $\delta>0$ be arbitrary, and define
\[
x'=1+\frac{\delta}{2}.
\]
Then
\[
|x'-1|=\frac{\delta}{2}<\delta,
\]
and $x'>1$. Since $z_0=e>0$ and $x'$ is real, all iterates lie on the
positive real axis. Along this orbit, the principal-value logarithm agrees with the
ordinary real logarithm.

The first iterate is
\[
z_1(x')=e^{x'}=e^{1+\delta/2}>e=r.
\]
Thus
\[
|z_1(x')|>r.
\]

We now show that the orbit of $x'$ diverges to $+\infty$. We claim that
\[
z_n(x')=e^{(x')^n}
\qquad\text{for all } n\ge0.
\]
For $n=0$, this is true because
\[
z_0=e=e^{(x')^0}.
\]
Assume the formula holds for some $n\ge0$. Then
\[
z_n(x')=e^{(x')^n}>0,
\]
so
\[
\log(z_n(x'))=\ln(z_n(x'))=(x')^n.
\]
Therefore
\[
z_{n+1}(x')
=
z_n(x')^{\,x'}
=
\exp(x'\log(z_n(x')))
=
\exp(x'(x')^n)
=
e^{(x')^{n+1}}.
\]
Thus the formula holds by induction.

Since $x'>1$, we have
\[
(x')^n\to+\infty
\qquad\text{as } n\to\infty.
\]
Hence
\[
z_n(x')=e^{(x')^n}\to+\infty.
\]
Therefore the orbit of $x'$ is unbounded, and
\[
x'\notin E(e,0).
\]

\smallskip
\noindent\textbf{Step 3: The point $x_0=1$ lies on the boundary of $E(e,0)$.}
Every neighborhood of $x_0=1$ contains $x_0$ itself, and we have already
shown that $x_0\in E(e,0)$. By Step 2, every neighborhood of $x_0=1$ also
contains some $x'\notin E(e,0)$. Therefore every neighborhood of $x_0=1$
intersects both $E(e,0)$ and its complement. Hence
\[
x_0=1\in\partial E(e,0).
\]
For completeness, let $U$ be any neighborhood of $x_0=1$ and choose
$y\in U\cap(0,1)$. The same induction as in Step 2 gives
\[
z_n(y)=e^{y^n}\longrightarrow1,
\]
so $y\in E(e,0)$.

\smallskip
\noindent\textbf{Step 4: The thresholded escape-time function is discontinuous at $x_0=1$.}
For $x_0=1$, the orbit satisfies $|z_n(1)|=e=r$ for every $n\ge0$.
Because the escape condition is strict, namely $|z_n|>r$, the orbit never
escapes past $r=e$. Thus
\[
T_r(1)=+\infty.
\]

Now let
\[
x_k=1+\frac{1}{k}
\qquad(k\ge1).
\]
Then $x_k\to1$. Also,
\[
z_1(x_k)=e^{1+1/k}>e=r,
\]
while
\[
z_0(x_k)=e=r
\]
does not strictly exceed $r$. Therefore
\[
T_r(x_k)=1
\qquad\text{for all } k\ge1.
\]
Thus $x_k\to1$, but
\[
T_r(x_k)=1
\qquad\text{for all } k\ge1,
\]
whereas
\[
T_r(1)=+\infty.
\]
In the stated order topology, the constant sequence $T_r(x_k)=1$ does not
converge to $+\infty=T_r(1)$. Hence $T_r$ is discontinuous at $x_0=1$ for
the chosen threshold $r=e$.
\end{proof}

\begin{remark}[Threshold dependence]
The discontinuity is specific to the strict threshold $r=e$. If $0<r<e$, then
$T_r\equiv0$ because $|z_0|=e$. If $r>e$, continuity of each finite set of
iterates near $x=1$ implies that, for every $N$, one has $T_r(x)>N$ sufficiently
close to $1$; hence $T_r$ is continuous at $1$ in the stated order topology.
\end{remark}

\smallskip

\begin{theorem}[Principal-Argument Vertical Boundedness Asymmetry at $x=2$]
\label{thm:principal-argument-vertical-asymmetry}
Consider the iteration
\[
z_{n+1}=z_n^{\,x}+c,\qquad z_0,c,x\in\C,
\]
with the principal-value complex-power convention and the zero convention of
Section~\ref{sec:definitions-framework}. In particular, for $z\ne0$,
\[
z^x=\exp(x\log z),\qquad
\log z=\ln|z|+i\operatorname{Arg}(z),\qquad \operatorname{Arg}(z)\in(-\pi,\pi].
\]
For
\[
z_0=i,\qquad c=0,\qquad x_0:=2,
\]
the following hold:
\begin{enumerate}
  \item $x_0\in E(i,0)$.
  \item For every $\varepsilon>0$, let
        $x_\pm=2\pm i\varepsilon$, and let
        $\{z_n^\pm\}_{n\ge0}$ denote the corresponding orbits. Then
        $x_+\in E(i,0)$, while $x_-\notin E(i,0)$. In fact,
        $|z_n^+|\to0$ and $|z_n^-|\to+\infty$.
  \item Hence $x_0\in\partial E(i,0)$.
\end{enumerate}
The proof exploits the principal-argument assignment $\operatorname{Arg}(-a)=\pi$ for $a>0$.
\end{theorem}

\begin{proof}
Since $c=0$, the iteration becomes
\[
z_{n+1}=z_n^{\,x}.
\]

\smallskip
\noindent\textbf{Step 1: boundedness at $x_0=2$.}
For $x=2$, we have
\[
z_1=i^2=-1,\qquad z_2=(-1)^2=1,\qquad z_n=1\quad(n\ge 2).
\]
Thus the orbit is bounded, so $x_0=2\in E(i,0)$.

\smallskip
\noindent\textbf{Step 2: first iterates for $x_\pm$.}
Let $\varepsilon>0$ and define
\[
x_+=2+i\varepsilon,\qquad x_-=2-i\varepsilon.
\]
Since
\[
\log(i)=\frac{i\pi}{2},
\]
we get
\[
x_\pm\log(i)
=(2\pm i\varepsilon)\frac{i\pi}{2}
=i\pi\mp\frac{\varepsilon\pi}{2}.
\]
Therefore
\[
z_1^\pm=i^{x_\pm}
=\exp(x_\pm\log(i))
=\exp(i\pi\mp\varepsilon\pi/2)
=-e^{\mp\varepsilon\pi/2}.
\]
Hence $z_1^+$ and $z_1^-$ both lie on the negative real axis. Although these
points lie on the branch cut, $\pi$ belongs to the principal-argument interval
$(-\pi,\pi]$. Therefore, by the principal-argument convention,
\[
\operatorname{Arg}(z_1^\pm)=\pi.
\]
Moreover,
\[
|z_1^+|=e^{-\varepsilon\pi/2}<1,
\qquad
|z_1^-|=e^{\varepsilon\pi/2}>1.
\]

\smallskip
\noindent\textbf{Step 3: modulus recursion.}
We first note that all iterates considered in this proof are nonzero. Indeed,
$z_0=i\ne 0$, and if $z_n^\pm\ne 0$, then
\[
z_{n+1}^\pm=(z_n^\pm)^{x_\pm}=\exp(x_\pm\log(z_n^\pm))\ne 0.
\]
Thus the claim follows by induction. We may therefore write
\[
z_n^\pm=\rho_n^\pm e^{i\theta_n^\pm},
\qquad
\rho_n^\pm=|z_n^\pm|,
\qquad
\theta_n^\pm=\operatorname{Arg}(z_n^\pm)\in(-\pi,\pi].
\]
Then
\[
\log(z_n^\pm)=\ln(\rho_n^\pm)+i\theta_n^\pm.
\]
For $x_+=2+i\varepsilon$, we obtain
\[
\rho_{n+1}^+
=|(z_n^+)^{x_+}|
=\exp(2\ln(\rho_n^+)-\varepsilon\theta_n^+).
\]
For $x_-=2-i\varepsilon$, we obtain
\[
\rho_{n+1}^-
=|(z_n^-)^{x_-}|
=\exp(2\ln(\rho_n^-)+\varepsilon\theta_n^-).
\]

\smallskip
\noindent\textbf{Step 4: boundedness for $x_+$.}
Since $\theta_n^+\in(-\pi,\pi]$, we have
\[
-\varepsilon\theta_n^+\le \varepsilon\pi.
\]
Therefore
\[
\rho_{n+1}^+
\le
e^{\varepsilon\pi}(\rho_n^+)^2.
\]
Using $\rho_1^+=e^{-\varepsilon\pi/2}$ and $\operatorname{Arg}(z_1^+)=\pi$, the recursion
gives
\[
\rho_2^+
=\exp(2\ln(\rho_1^+)-\varepsilon\pi)
=\exp(-2\varepsilon\pi)
=e^{-2\varepsilon\pi}
<e^{-\varepsilon\pi}.
\]
This value starts the monotone estimate. We now prove by induction that, for
all $n\ge 2$,
\[
\rho_n^+\le e^{-2\varepsilon\pi}.
\]
The base case $n=2$ holds by the computation above. If
$\rho_n^+\le e^{-2\varepsilon\pi}$, then
\[
\rho_{n+1}^+
\le e^{\varepsilon\pi}(\rho_n^+)^2
\le e^{\varepsilon\pi}e^{-2\varepsilon\pi}\rho_n^+
=e^{-\varepsilon\pi}\rho_n^+.
\]
In particular,
\[
\rho_{n+1}^+\le e^{-\varepsilon\pi}\rho_n^+\le \rho_n^+
\le e^{-2\varepsilon\pi}.
\]
Thus the induction closes. Moreover,
\[
\rho_{n+1}^+\le e^{-\varepsilon\pi}\rho_n^+
\qquad(n\ge 2),
\]
so $\rho_n^+\to 0$ geometrically. Hence the orbit for $x_+$ is bounded, and
therefore
\[
x_+\in E(i,0).
\]

\smallskip
\noindent\textbf{Step 5: divergence for $x_-$.}
Since $\theta_n^-\in(-\pi,\pi]$, we have
\[
\varepsilon\theta_n^-\ge -\varepsilon\pi.
\]
Therefore
\[
\rho_{n+1}^-
\ge
e^{-\varepsilon\pi}(\rho_n^-)^2.
\]
Using $\rho_1^-=e^{\varepsilon\pi/2}$ and $\operatorname{Arg}(z_1^-)=\pi$, the recursion
gives
\[
\rho_2^-
=\exp(2\ln(\rho_1^-)+\varepsilon\pi)
=\exp(2\varepsilon\pi)
=e^{2\varepsilon\pi}.
\]
Let
\[
s_n=\ln(\rho_n^-).
\]
Taking logarithms, for $n\ge 2$ we obtain
\[
s_{n+1}\ge 2s_n-\varepsilon\pi.
\]
Since $s_2=2\varepsilon\pi$, we claim that
\[
s_n\ge \varepsilon\pi+2^{n-2}\varepsilon\pi
\qquad(n\ge 2).
\]
The claim is true for $n=2$. If it holds for some $n\ge 2$, then
\[
s_{n+1}
\ge 2s_n-\varepsilon\pi
\ge 2(\varepsilon\pi+2^{n-2}\varepsilon\pi)-\varepsilon\pi
=\varepsilon\pi+2^{n-1}\varepsilon\pi.
\]
Thus the claim follows by induction. Therefore $s_n\to+\infty$, so
\[
\rho_n^-=|z_n^-|\to+\infty.
\]
Hence the orbit for $x_-$ diverges in modulus, and
\[
x_-\notin E(i,0).
\]

\smallskip
\noindent\textbf{Step 6: boundary conclusion.}
Let $U$ be any neighborhood of $x_0=2$. Choose $\varepsilon>0$ small enough
that both $2+i\varepsilon$ and $2-i\varepsilon$ lie in $U$. By the results
above, $2+i\varepsilon\in E(i,0)$ while $2-i\varepsilon\notin E(i,0)$.
Thus every neighborhood of $x_0=2$ contains both points of $E(i,0)$ and points
outside $E(i,0)$. Consequently,
\[
x_0=2\in\partial E(i,0).
\]

The branch-sensitive step is the evaluation of the negative real first
iterates $z_1^\pm$, for which the principal convention gives
$\theta_1^\pm=\pi$. More generally, for
$z_n=\rho_n e^{i\theta_n}$ and $x=a+ib$, the real part of the exponent is
\[
\operatorname{Re}(x\log(z_n))
=a\ln(\rho_n)-b\theta_n.
\]
At $n=1$, the choice $\theta_1^\pm=\pi$ yields the decisive values
$\rho_2^+=e^{-2\varepsilon\pi}$ and
$\rho_2^-=e^{2\varepsilon\pi}$. The global principal-argument range
$\theta_n^\pm\in(-\pi,\pi]$ then supplies the estimates used above. Thus the
principal-argument convention is an explicit input to the directional
boundedness conclusion.
\end{proof}

\medskip

\section{Finite-Time Numerical Illustrations}
\label{sec:numerical-observations}

This section records finite-time computations in the exponent plane. Unless
otherwise stated, the initial value and additive parameter are fixed at
\[
z_0=c=\frac12.
\]
Each image records one of four finite-time outcomes: the first threshold
crossing occurs on an iteration from $1$ through $50$; no crossing occurs
through iteration $50$; the orbit becomes formally undefined before a
crossing; or floating-point evaluation fails before a formal outcome is
obtained. The figures therefore describe finite-time threshold outcomes rather
than exact membership in $E(1/2,1/2)$. In particular, failure to cross a
threshold within $50$ iterations does not imply that the orbit remains bounded
for all time.

\subsection{Computational Setup}

For each sampled exponent $x$, the orbit
\[
z_{n+1}=z_n^{\,x}+c
\]
was evaluated using the principal-argument convention from
Section~\ref{sec:definitions-framework}. For a maximum iteration count
$N_{\max}\in\Z_{\ge0}$, define
\[
\mathcal{W}_{N_{\max}}(z_0,c)
=
\left\{
x\in\C:
z_0,z_1,\ldots,z_{N_{\max}}
\text{ are all defined}
\right\}.
\]
For an escape threshold $r>0$, define the formal finite-time outcome
\[
\mathcal{S}_{r,N_{\max}}
:
\C\longrightarrow
\{0,1,\ldots,N_{\max}\}\cup\{\mathsf{NC},\mathsf{U}\}.
\]
The value is $\mathcal{S}_{r,N_{\max}}(x)=n$ if $z_0,\ldots,z_n$ are
defined, $|z_j|\le r$ for $0\le j<n$, and $|z_n|>r$. It is
$\mathsf{NC}$ if $x\in\mathcal{W}_{N_{\max}}(z_0,c)$ and no
threshold crossing occurs through iteration $N_{\max}$, and it is
$\mathsf{U}$ if the orbit becomes undefined by that cutoff before any
crossing. Thus $\mathsf{NC}$ means no crossing through the cutoff, while
$\mathsf{U}$ means formally undefined.

Floating-point evaluation has one additional outcome, $\mathsf{F}$, for a
numerical failure or nonfinite result that prevents the renderer from assigning
a formal outcome. Every figure in this section uses $N_{\max}=50$. Because
$|z_0|=0.5<r$ for every displayed panel, its numerical crossing labels range
from $1$ through $50$ rather than including $0$.

The same fixed first-crossing scale is used in every panel. For a crossing
iteration $k\in\{1,\ldots,50\}$, the renderer applies Mathematica's
\emph{Rainbow} color map at $((k-1)/49)^{0.4}$. The endpoint $k=50$ therefore
means a genuine first crossing on iteration $50$. The special outcomes are
colored separately: $\mathsf{NC}$ is black, $\mathsf{U}$ is medium gray,
and $\mathsf{F}$ is magenta. Internally these three outcomes are stored as
$-3,-1$, and $-2$, respectively, and are status codes rather than iteration
counts. Equal colors therefore represent equal first-crossing iterations or
equal special outcomes across panels.

The corresponding visual key is shown below.
\begin{center}
\includegraphics[width=0.94\textwidth]{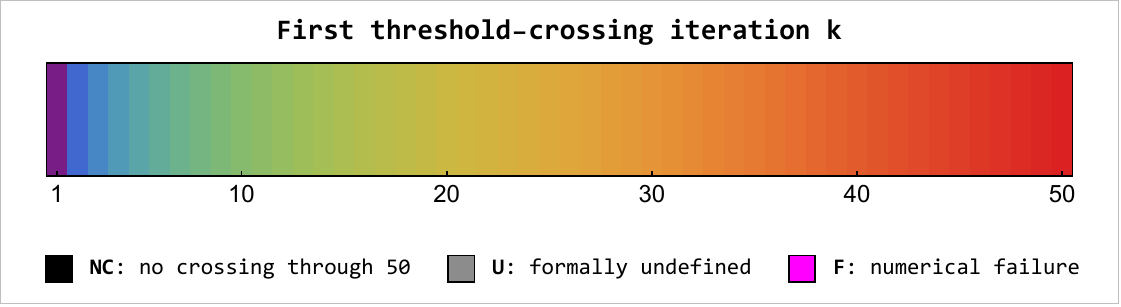}
\end{center}

For $x_c\in\C$ and $L>0$, let
\[
Q(x_c,L)
=
\left\{
x\in\C:
\left|\operatorname{Re}(x-x_c)\right|\le\frac{L}{2},
\quad
\left|\operatorname{Im}(x-x_c)\right|\le\frac{L}{2}
\right\}.
\]
Thus $Q(x_c,L)$ is the square viewport centered at $x_c$ with side length $L$.
The numerical parameters for the six illustrations are listed in
Table~\ref{tab:numerical-parameters}.

\begin{table}[!htbp]
\caption{Parameters for the six finite-time illustrations. Every panel uses
$z_0=c=0.5$, $N_{\max}=50$, and a $2160\times2160$ grid of exponents.}
\label{tab:numerical-parameters}
\centering
\footnotesize
\begin{tabular}{@{}lccc@{}}
\hline
Illustration & Center $x_c$ & Side length $L$ & Threshold $r$ \\
\hline
Threshold comparison (a) & $-2.5$ & $8$ & $5$ \\
Threshold comparison (b) & $-2.5$ & $8$ & $10$ \\
Threshold comparison (c) & $-2.5$ & $8$ & $20$ \\
Threshold comparison (d) & $-2.5$ & $16$ & $100$ \\
Adjacent regions & $-2.1$ & $2$ & $5$ \\
Multi-region zoom & $2.2372+2.2948i$ & $1/64$ & $5$ \\
\hline
\end{tabular}
\end{table}

\noindent\textbf{Implementation and numerical checks.}
The figures were generated in Mathematica~14.3 using double-precision
arithmetic and the strict threshold test
\[
|z_n|^2>r^2.
\]
Each panel samples a $2160\times2160$ grid in the exponent plane. The complete
rendering code, plotting controls, viewport parameters, and color convention
are provided in the accompanying Wolfram Language source file
\texttt{Eset\_Renderer.wl}. As a consistency check, $81$ grid points from each
panel were recomputed on the CPU; all $486$ classifications matched the
accelerated computation.

For a nonzero base, the renderer evaluates the principal argument numerically
using \texttt{atan2}. At an exact zero base, it returns $0$ only when
$\operatorname{Re}(x)>0$ and otherwise returns the formally undefined status
$\mathsf{U}$, in accordance with Section~\ref{sec:definitions-framework}.
Nonfinite arithmetic and underflow from a mathematically nonzero value are
assigned the numerical-failure status $\mathsf{F}$. None of the six
displayed grids produced a $\mathsf{U}$ or $\mathsf{F}$ status.

\subsection{Observation 1: Cross-Panel Variation}

Figure~\ref{fig:escape-threshold-comparison} compares four finite-time
renderings obtained with different thresholds and viewports.

\begin{figure}[H]
    \centering
    \begin{tabular}{@{}c@{\hspace{0.025\textwidth}}c@{}}
        \includegraphics[width=0.4\textwidth]{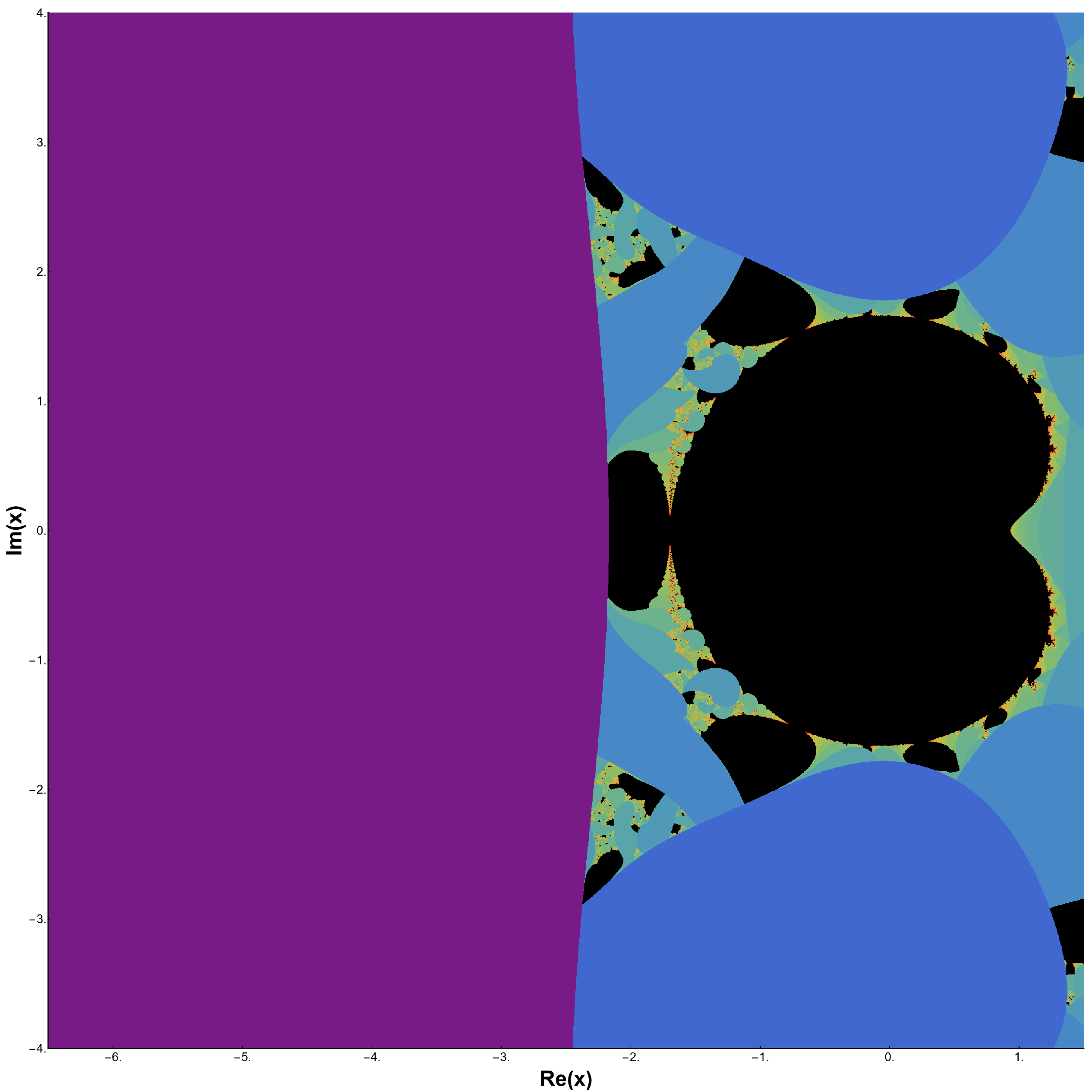} &
        \includegraphics[width=0.4\textwidth]{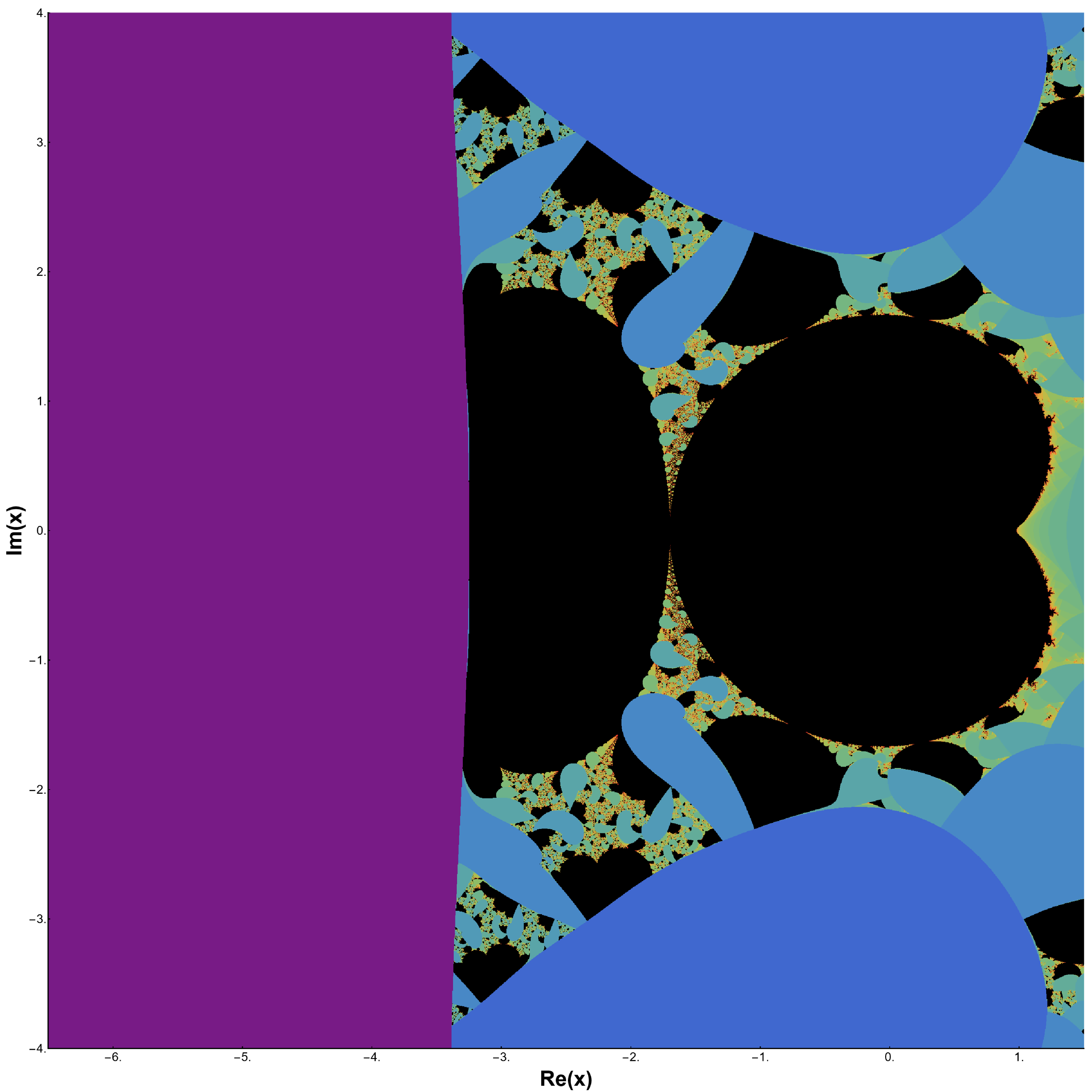} \\
        {\small (a) $r=5$} & {\small (b) $r=10$} \\[0.3em]
        \includegraphics[width=0.4\textwidth]{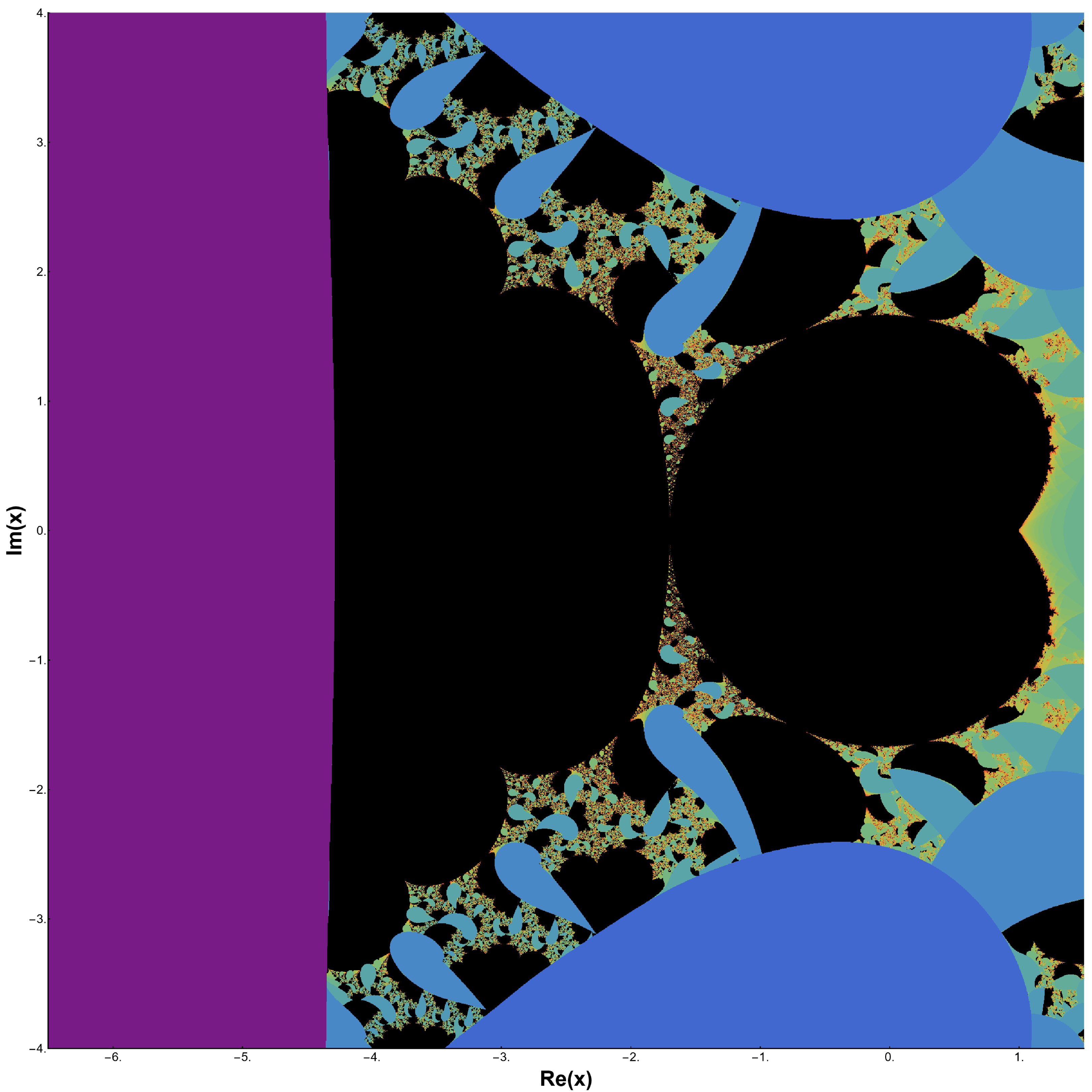} &
        \includegraphics[width=0.4\textwidth]{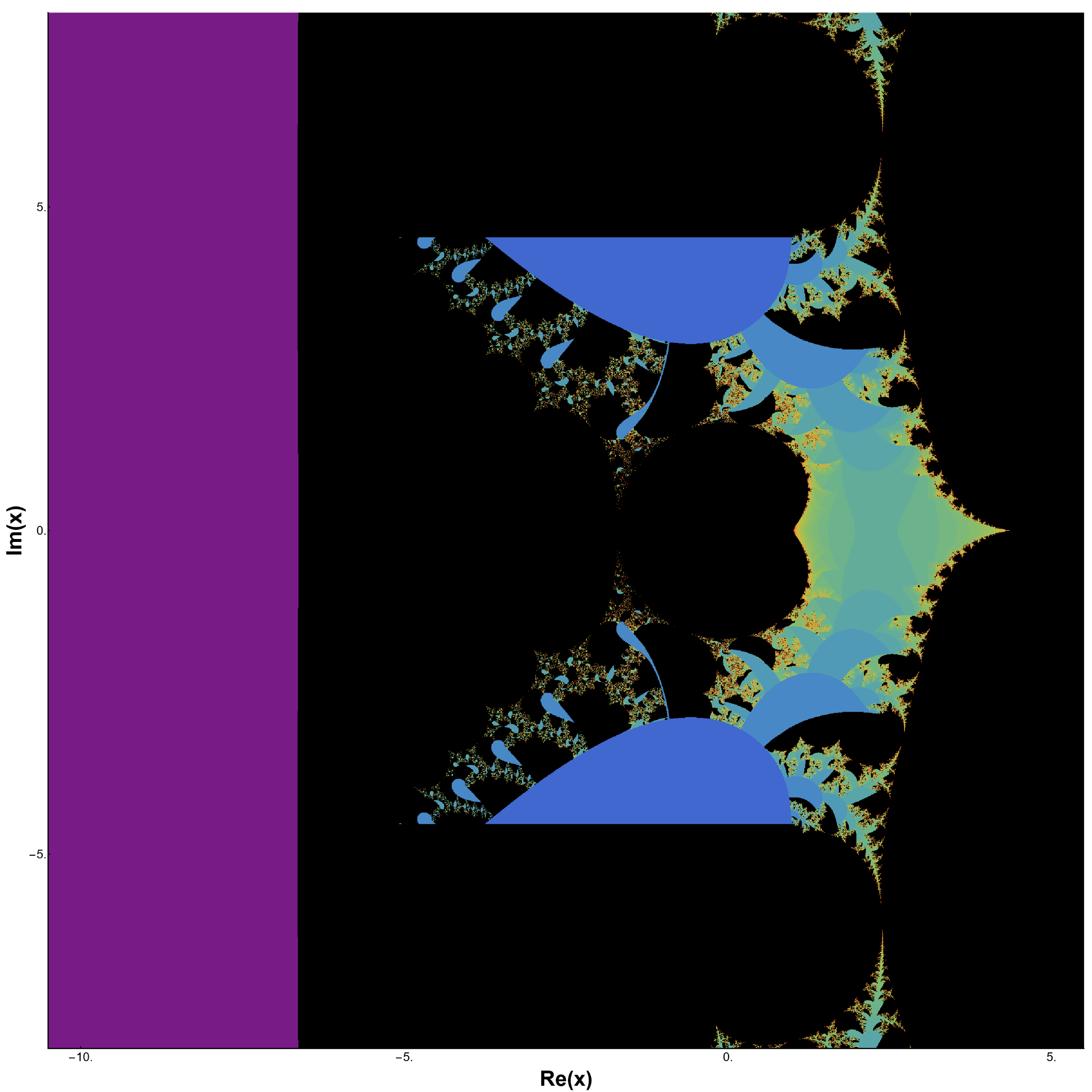} \\
        {\small (c) $r=20$} & {\small (d) $r=100$}
    \end{tabular}
    \caption{Finite-time threshold classifications for $z_0=c=0.5$ and
$N_{\max}=50$. Panels (a)--(c) use $Q(-2.5,8)$ with thresholds
$r=5,10,20$, respectively; panel (d) uses $Q(-2.5,16)$ with $r=100$.
Rainbow colors encode first-crossing iterations $1$--$50$; black denotes
no crossing through iteration $50$, medium gray denotes a formally undefined
orbit, and magenta denotes numerical failure.}
    \label{fig:escape-threshold-comparison}
    \label{fig:E-set_R5}
    \label{fig:E-set_R10}
    \label{fig:E-set_R20}
    \label{fig:E-set_R100}
\end{figure}

\FloatBarrier

\noindent\textbf{Observation.}
The four renderings visibly differ. Across panels, some regions receive
different finite-time classifications or first-crossing colors, and the
visible layered and oval-shaped structures change. Panels (a)--(c) use
identical settings except for $r$, so their differences arise from
changing the escape threshold. Panel (d) also uses a different viewport and
must therefore be interpreted separately.

\FloatBarrier
\subsection{Observation 2: Adjacent Finite-Time Outcome Regions}

Figure~\ref{fig:conver_diver} shows black finite-time non-crossing regions
directly adjacent to Rainbow first-crossing regions in the exponent plane.

\begin{figure}[htbp]
    \centering
    \includegraphics[width=0.45\textwidth]{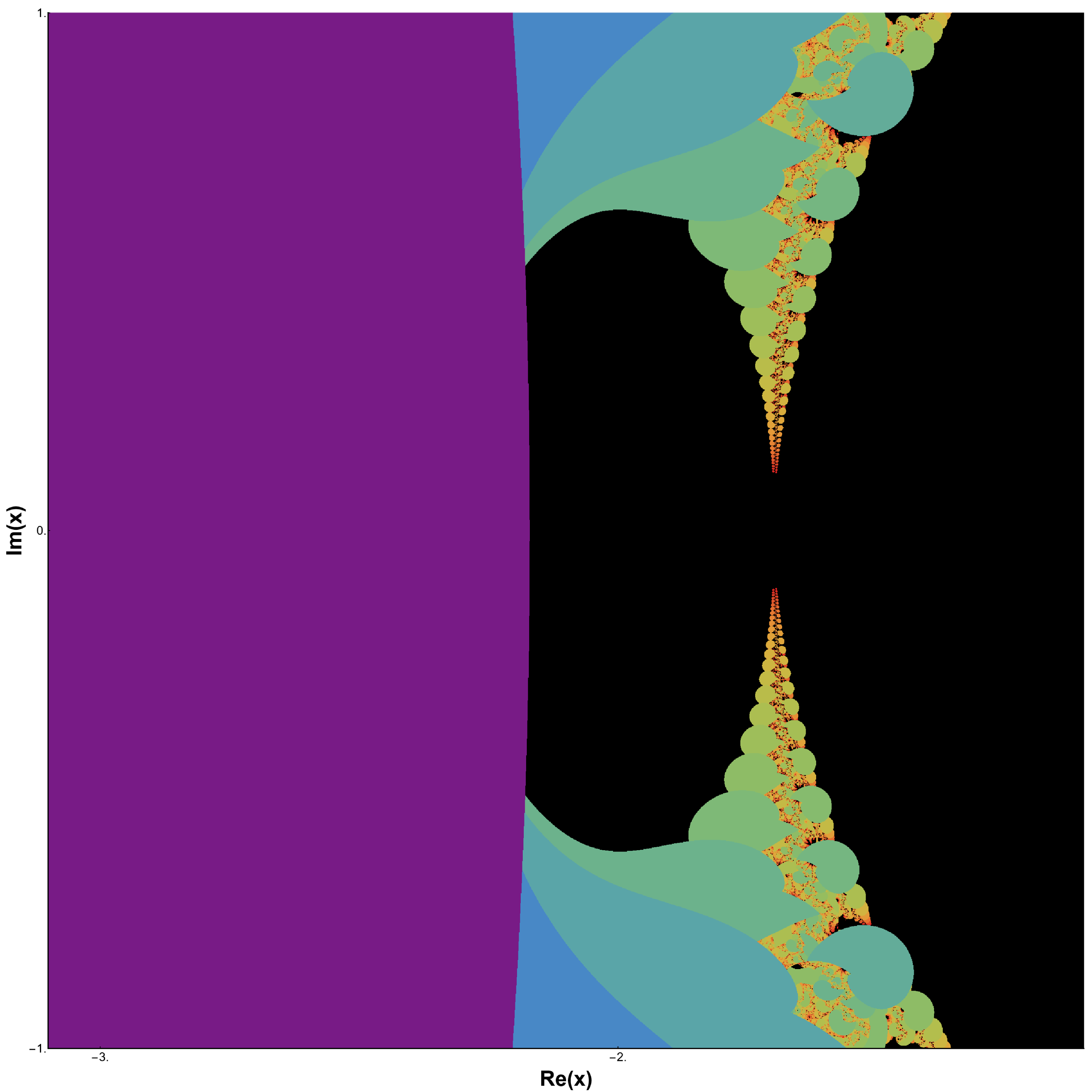}
    \caption{A finite-time color rendering of the exponent plane on the square
    viewport $[-3.1,-1.1]\times[-1,1]$, with $z_0=c=0.5$,
    $N_{\max}=50$, $r=5$, and a $2160\times2160$ grid. Rainbow colors encode
    first-crossing iterations $1$--$50$; black denotes no crossing through
    iteration $50$, medium gray denotes a formally undefined orbit, and magenta
    denotes numerical failure.}
    \label{fig:conver_diver}
\end{figure}
\FloatBarrier

\noindent\textbf{Observation.}
The numerical rendering shows sharp interfaces between black $\mathsf{NC}$
regions and Rainbow first-crossing regions, as well as between different
first-crossing colors. Neighboring displayed samples can therefore receive
different finite-time outcomes or crossing iterations. The $\mathsf{U}$ and
$\mathsf{F}$ statuses do not occur in this panel. No all-time calculation is
made, so these finite-time outcomes are not treated as verified E-Set
classifications. The visual adjacency is qualitatively reminiscent of
Theorem~\ref{thm:boundary-escape-time-instability}, which concerns the
specific family $E(e,0)$ at $x=1$ with the strict threshold $r=e$.

\FloatBarrier
\subsection{Observation 3: Multi-Region Pattern}

Figure~\ref{fig:trifecta} shows several Rainbow first-crossing colors meeting a
black $\mathsf{NC}$ region near a common apparent interface.

\begin{figure}[htbp]
    \centering
    \includegraphics[width=0.45\textwidth]{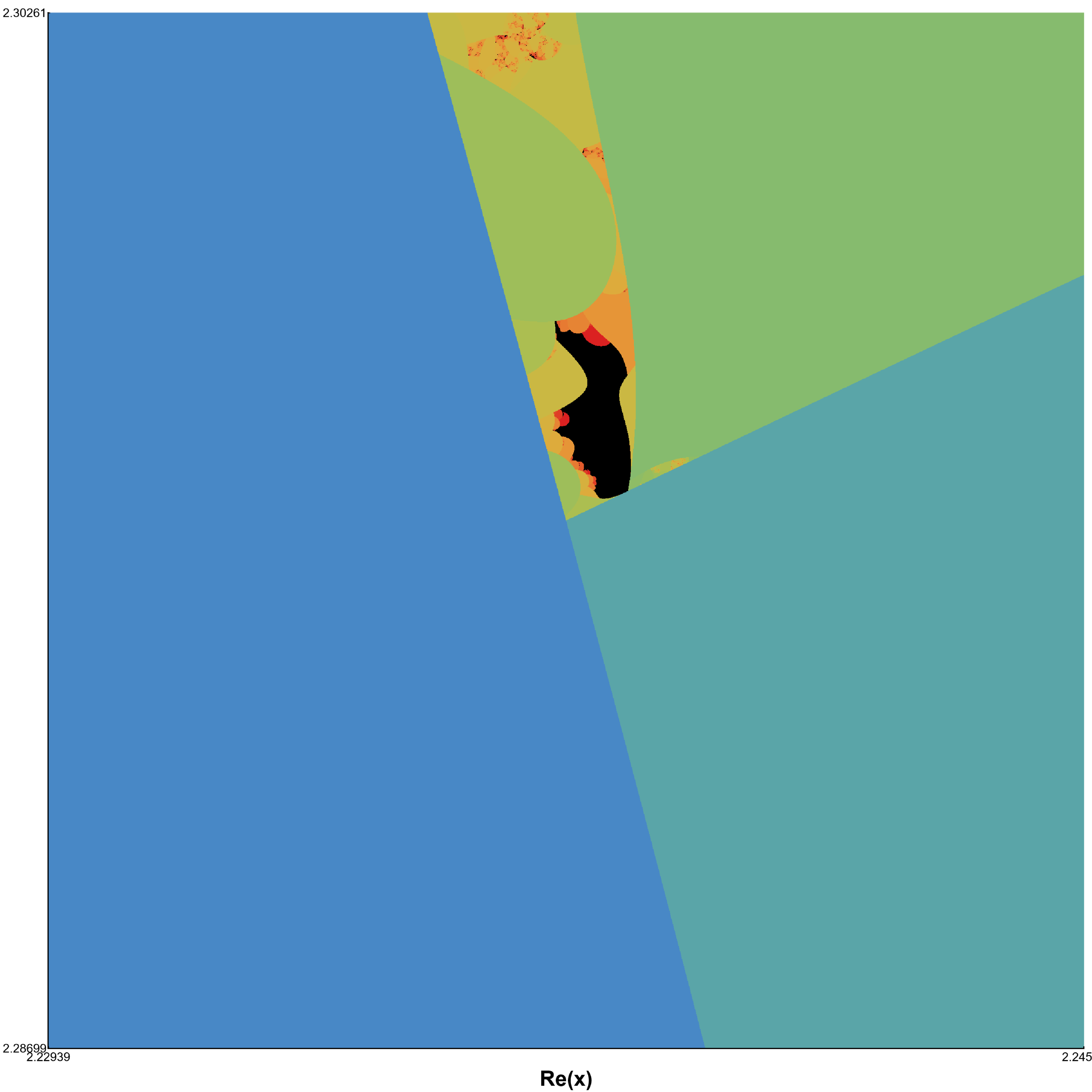}
    \caption{Finite-time color rendering of the exponent plane on
    $Q(2.2372+2.2948i,1/64)$, with $z_0=c=0.5$, $N_{\max}=50$, $r=5$,
    and a $2160\times2160$ grid. Rainbow colors
    encode first-crossing iterations $1$--$50$; black denotes no crossing
    through iteration $50$, medium gray denotes a formally undefined orbit,
    and magenta denotes numerical failure.}
    \label{fig:trifecta}
\end{figure}
\FloatBarrier

\noindent\textbf{Observation.}
The rendering contains a single apparent junction at which a black
$\mathsf{NC}$ region and three distinct first-crossing color regions meet.
Thus nearby sampled exponents exhibit no crossing through iteration $50$ or
first crossings at three different iterations. Figure~\ref{fig:trifecta}
therefore shows a finite-time pattern more complicated than a single smooth
interface between one first-crossing region and the non-crossing region.

\FloatBarrier
\subsection{Observation 4: Branch-Sensitive Term}

The numerical figures show asymmetric finite-time structures, but the exported
outcome data do not record whether any state-plane iterate lies on the negative
real axis.

\noindent\textbf{Observation.}
For $z_n\ne0$ and $x=a+ib$, the identity
\[
\ln|z_{n+1}-c|
=
a\ln|z_n|-b\operatorname{Arg}(z_n)
\]
isolates the branch-sensitive term. In
Theorem~\ref{thm:principal-argument-vertical-asymmetry}, an iterate lies exactly
on the negative real axis, so the assignment $\operatorname{Arg}(-\rho)=\pi$ enters the
proof. Because no analogous negative-real contact is recorded for the six
numerical illustrations, no branch-cut mechanism is inferred from them.

\subsection{Summary of the Finite-Time Illustrations}

The six illustrations record three finite-time patterns. First, the four
threshold-labeled renderings differ visibly; panels (a)--(c) form a controlled
threshold comparison, while panel (d) also uses a wider viewport. Second,
black regions representing no crossing through iteration $50$ occur next to
Rainbow first-crossing regions, and neighboring samples can have different first-crossing times.
Third, the images contain multi-region and asymmetric finite-time structures.

These finite-time classifications do not determine E-Set membership or
all-time divergence. The separate all-time questions are stated next.

\FloatBarrier

\section{Open Problems and Conjectures}
\label{sec:open-problems-conjectures}

The following conjectures formulate all-time questions suggested by the
theoretical results and numerical illustrations.

The first conjecture is motivated by
Theorem~\ref{thm:no-universal-radius-unit-circle}, which proves the failure of a
universal escape radius for $(z_0,c)=(i,0)$. It asks whether that failure
persists under small additive perturbations.

\begin{conjecture}[Local Additive Persistence of Escape-Radius Failure]
\label{conj:broad-no-universal-escape}
There exists $\delta>0$ such that, for every $c\in\C$ with $|c|<\delta$, the
E-Set $E(i,c)$ admits no finite universal escape radius. More precisely, for
every such $c$ and every $r>0$, there exist an exponent
$x=x(r,c)\in E(i,c)$ and an index $n\ge0$ such that
\[
|z_n|>r,
\]
while the orbit remains bounded.
\end{conjecture}

This conjecture asks for one uniform neighborhood of $c=0$ with the initial
value fixed at $z_0=i$, not for all E-Sets or for a generic parameter class.
The exponent and bounded orbit may depend on the prescribed radius and on $c$.

For the remaining conjectures, define
\[
\mathcal{W}(z_0,c)
=
\{x\in\C:\text{the orbit is well-defined for every }n\ge0\}
\]
and
\[
\mathcal{D}_{\infty}(z_0,c)
=
\{x\in\mathcal{W}(z_0,c):|z_n(x)|\to+\infty\}.
\]
For $R>0$ and $x\in\mathcal{W}(z_0,c)$, also set
\[
T_R(x)=\inf\{n\ge0:|z_n(x)|>R\},
\]
with the infimum of the empty set equal to $+\infty$. Unlike the finite-time
outcome $\mathcal{S}_{r,N_{\max}}$, this notation imposes no iteration cutoff.

Figure~\ref{fig:conver_diver} displays adjacent finite-time outcome regions
and motivates testing whether that adjacency persists as an all-time boundary
phenomenon. Divergent access is a separate conjectural requirement and is not
inferred from the rendering.

\begin{conjecture}[Divergently Accessible Boundary Continuum]
\label{conj:boundary_divergence}
For $z_0=c=1/2$, there exists a compact connected set
$\Gamma\subset\partial E(1/2,1/2)$ containing at least two points such that,
for every $x\in\Gamma$ and every $\delta>0$, there exists
$y\in\mathcal{D}_{\infty}(1/2,1/2)$ with
\[
0<|y-x|<\delta.
\]
\end{conjecture}

Thus every point of the conjectured continuum would be a true E-Set boundary
point approximable by distinct all-time divergent parameters. The conjecture
is limited to the existence of one such continuum for one parameter pair; it
does not assert divergent access at every boundary point or identify the
sampled interface in Figure~\ref{fig:conver_diver} with $\Gamma$.

The multi-region pattern in Figure~\ref{fig:trifecta} suggests testing whether
three all-time regimes can accumulate at one true boundary point. The finite
cutoff neither identifies such a point nor classifies the required orbits.

\begin{conjecture}[Multi-Regime Boundary Point]
\label{conj:multi_sector_boundary}
For $z_0=c=1/2$ and the strict threshold $R=5$, there exist a point
$x_*\in\partial E(1/2,1/2)$ and sequences of points
\[
b_k\to x_*,
\qquad
q_k\to x_*,
\qquad
d_k\to x_*,
\]
whose terms are different from $x_*$ and satisfy
\begin{enumerate}
    \item $b_k\in E(1/2,1/2)$ for every $k$;
    \item $q_k,d_k\in\mathcal{D}_{\infty}(1/2,1/2)$ for every $k$;
    \item $T_5(q_k)\le5$ for every $k$;
    \item $T_5(d_k)\to+\infty$ as $k\to\infty$.
\end{enumerate}
\end{conjecture}

The three regimes are bounded orbits, divergent orbits with a uniformly early
first crossing of $5$, and divergent orbits whose first crossing of $5$ is
arbitrarily delayed. Establishing these sequences requires independent
all-time boundedness and divergence estimates.

The last conjecture is motivated by
Theorem~\ref{thm:principal-argument-vertical-asymmetry}, which proves a
principal-argument-dependent vertical asymmetry for $c=0$ and an exact
negative-real iterate. It asks whether an analogous mechanism occurs in an
additive family. The numerical data record no branch-cut contacts and are not
evidence for this mechanism.

\begin{conjecture}[Additive Branch-Cut-Anchored Vertical Asymmetry]
\label{conj:generic_branch_cut}
There exist $z_0,c\in\C$ with $c\ne0$, a point
$x_*\in E(z_0,c)\cap\partial E(z_0,c)$, an index $m\ge1$, a sign
$\sigma\in\{-1,1\}$, and $\varepsilon_0>0$ such that
\[
z_m(x_*)\in(-\infty,0)
\]
and, for every $0<\varepsilon<\varepsilon_0$,
\[
x_*+\sigma i\varepsilon\in E(z_0,c),
\qquad
x_*-\sigma i\varepsilon\in\mathcal{D}_{\infty}(z_0,c).
\]
\end{conjecture}

This statement gives ``asymmetry'' an exact local meaning and anchors the
principal-argument term at a negative-real iterate. It asserts existence with
$c\ne0$, not genericity. Whether such triples form an open or otherwise
prevalent parameter family, and whether changing the argument convention reverses
or removes the asymmetry, remain separate open problems.

\section{Conclusion and Future Work}

This paper studied the exponent-coordinate fiber $E(z_0,c)$ of the full
boundedness locus $\mathcal{B}\subset\C^3$ for
\[
z_{n+1}=z_n^{\,x}+c.
\]
Together, $M(z_0,x)$, $J(c,x)$, and $E(z_0,c)$ represent the $c$-, $z_0$-,
and $x$-coordinate fiber families of $\mathcal{B}$, respectively.

The main theoretical results establish the failure of a finite universal escape
radius in several explicit exponent-parameter families. For the pure-power
slice $c=0$, we proved that no finite universal escape radius exists on the
unit-circle family $|z_0|=1$, $z_0\ne 1$, and on the full real axis
$z_0\in\R\setminus\{0,1\}$. We also proved an additive example with
$c\ne0$ in which both $z_0$ and $c$ have nonzero real and imaginary parts.

We also proved two parameter-specific phenomena. For $E(e,0)$, at the boundary
point $x=1$ and for the strict threshold $r=e$, the extended-valued escape-time
function $T_e$ is discontinuous. For $E(i,0)$, at $x=2$, the perturbations
$2+i\varepsilon$ and $2-i\varepsilon$ produce opposite orbit behavior for
every $\varepsilon>0$: the orbit corresponding to $2+i\varepsilon$ tends to
$0$ in modulus, whereas the orbit corresponding to $2-i\varepsilon$ diverges.
The latter proof uses the principal-argument assignment
$\operatorname{Arg}(-a)=\pi$ on the
negative real axis. Thus the principal-value definition of complex
exponentiation can enter the boundary structure of the exponent-parameter plane
as a dynamical input.

The six numerical illustrations record cross-panel variation, adjacency
between non-crossing and first-crossing regions, multi-region finite-time
patterns, and visible asymmetry. Table~\ref{tab:numerical-parameters} and the
figure captions record the computational parameters: panels (a)--(c) isolate
the effect of changing the threshold, whereas panel (d) also uses a wider
viewport. The adjacent-region and multi-region renderings motivate questions
about all-time behavior, but finite-time classifications alone do not determine
E-Set membership. Conjecture~\ref{conj:generic_branch_cut} is instead motivated
by the proved principal-argument asymmetry, since the numerical data record no
branch-cut contacts.

Several directions remain open for future work.

\subsection{Real-Coordinate Slices}

Let $\mathcal{B}\subset\C^3$ be the full boundedness locus defined in
Section~\ref{sec:definitions-framework}.
For an affine embedding $\Phi:\R^2\to\C^3$ obtained by fixing four of the six
real coordinates, define
\[
S_{\Phi}=\{u\in\R^2:\Phi(u)\in\mathcal{B}\}.
\]
If the two varying coordinates are the real and imaginary parts of $x$, $c$,
or $z_0$, then under the standard identification $\R^2\cong\C$, $S_{\Phi}$ is
respectively identified with an E-Set, an $M(z_0,x)$ set, or a $J(c,x)$ set as
defined in Section~\ref{sec:definitions-framework}. A slice that mixes
coordinates from different complex variables is instead a cross-section of
$\mathcal{B}$ and is not automatically a canonical dynamical object.

A concrete program is to specify $\Phi$ and determine whether $S_{\Phi}$ is
empty, bounded, connected, or disconnected, and how its boundary meets the loci
where a finite iterate equals $0$ or lies on the negative real axis. A first
technical target is a finite-iterate regularity lemma: for fixed $N$, determine
the continuity and local real-analyticity of
\[
u\longmapsto
\bigl(z_0(\Phi(u)),z_1(\Phi(u)),\ldots,z_N(\Phi(u))\bigr)
\]
on regions where the preceding iterates remain a positive distance from $0$
and the negative-real cut of the principal-value logarithm. The zero-contact
and cut-contact loci must be
analyzed separately, since the regularity conclusion can fail there.

For a numerical slice, record the embedding $\Phi$, viewport, grid, escape threshold,
iteration limit, and handling of undefined orbits. Such
data may identify candidate components and boundary points, but membership in
$S_{\Phi}$ still requires an all-time trapping argument; finite-time
non-escape is insufficient.

\subsection{Proof Programs for the Conjectures}

For fixed $(z_0,c)$, define the bounded-excursion invariant
\[
A(z_0,c)
=
\sup\{|z_n(x)|:x\in E(z_0,c),\ n\ge0\},
\]
with the supremum over the empty set taken to be $0$. Directly from the
definition,
\[
\begin{gathered}
E(z_0,c)\text{ admits a finite universal escape radius}\\
\Longleftrightarrow\quad A(z_0,c)<+\infty.
\end{gathered}
\]
Indeed, a number $r>0$ is universal exactly when $|z_n(x)|\le r$ for every
$x\in E(z_0,c)$ and every $n\ge0$. The escape-radius classification problem is
therefore to determine the pairs for which this invariant is finite. The
current theorems establish $A(z_0,c)=+\infty$ for the stated pure-power
families and the additive example with nonzero real and imaginary parts. The
trivial pure-power pairs satisfy
\[
A(0,0)=0,
\qquad
A(1,0)=1;
\]
no further pair is classified here. A proof of $A(z_0,c)=+\infty$ must allow the
exponent and bounded orbit to depend on the prescribed radius; a proof of
finiteness requires one bound valid for every $x\in E(z_0,c)$.

The four conjectures above give separate proof programs:
\begin{enumerate}
    \item For Conjecture~\ref{conj:broad-no-universal-escape}, begin with the
    exact-collapse construction for $E(i,0)$. The target is a number $\delta>0$
    and estimates uniform for $|c|<\delta$ that, for each prescribed radius,
    produce a large excursion followed by an all-time trapping pattern. The
    exponent may depend on $c$ and the radius, but the neighborhood size may not;
    zero encounters and every principal argument must remain controlled.

    \item For Conjecture~\ref{conj:boundary_divergence}, the target is an
    explicit candidate $\Gamma$ together with proofs of compactness and
    connectedness, two-sided approximation showing
    $\Gamma\subset\partial E(1/2,1/2)$, and distinct divergent parameters
    approaching every point of $\Gamma$. Certified finite-iterate enclosures may
    locate a candidate, but they cannot establish any of these all-time or
    topological conditions alone.

    \item For Conjecture~\ref{conj:multi_sector_boundary}, one must construct the
    three stated sequences and prove their all-time behavior. Certified
    finite-iterate enclosures may locate candidates and bound $T_5$, but they
    must be combined with independent trapping and divergence estimates.

    \item For Conjecture~\ref{conj:generic_branch_cut}, one must first produce an
    additive orbit with the specified negative-real iterate and then derive
    uniform one-sided boundedness and divergence estimates for all sufficiently
    small vertical perturbations. Comparing the branches $\log_{\alpha,k}$ on
    specified cut domains is a separate question and is not implied by the
    principal-argument conjecture.
\end{enumerate}

These programs keep $E(z_0,c)$, $\mathcal{D}_{\infty}(z_0,c)$, and the set of
parameters with undefined orbits distinct. In particular, a point outside an
E-Set need not have an orbit that tends to infinity. Questions about
connectedness, symmetry, or local self-similarity should likewise be posed for
a specified E-Set or slice and with a stated topological or scaling property,
rather than inferred from a finite rendering.

\subsection{Higher-Dimensional Parameter Spaces}

For fixed $z_0\in\C$, separate the all-time well-definedness locus from the
boundedness locus by setting
\[
\mathcal{W}_{z_0}
=
\{(c,x)\in\C^2:
\text{the orbit is well-defined for every }n\ge0\}
\]
and
\[
\mathcal{P}_{z_0}
=
\{(c,x)\in\mathcal{W}_{z_0}:
\sup_{n\ge0}|z_n|<+\infty\}.
\]
For fixed $c$, the $x$-fiber of $\mathcal{P}_{z_0}$ is $E(z_0,c)$; for fixed
$x$, its $c$-fiber is $M(z_0,x)$. Precise higher-dimensional questions include
the topology of these fibers, their variation with the fixed parameter, and the
projections of $\mathcal{P}_{z_0}$ onto the $c$- and $x$-planes.

The case $z_0=0$ is fixed exactly by the zero convention. The first iterate is
defined if and only if $\operatorname{Re}(x)>0$; once this holds, every later
zero is also covered by the same convention and every nonzero power is defined
by the principal-value logarithm. Hence
\[
\mathcal{W}_0
=
\{(c,x)\in\C^2:\operatorname{Re}(x)>0\},
\]
and
\[
\mathcal{P}_0
=
\{(c,x)\in\C^2:
\operatorname{Re}(x)>0
\text{ and the orbit is bounded}\}.
\]
The complementary half-space must be classified as undefined, not escaped.

Dynamically defined slices can be specified by an orbit relation
\[
z_p=z_q,
\qquad
0\le p<q,
\]
provided every preceding iterate is well-defined, or by requiring a specified
iterate to equal $0$ or to lie on the negative real axis. A concrete program is
to determine the regular portions and singular zero- or cut-contact portions of
these relation loci, then study their intersections with $\mathcal{P}_{z_0}$.
Two-dimensional fibers and certified finite computations may supply candidates,
but projections of finite-time classifications do not establish boundedness in
the full four-real-dimensional locus.

\bibliographystyle{unsrt}
\bibliography{references}

\end{document}